\documentclass[a4paper]{amsart}

\usepackage[T1]{fontenc}
\usepackage{amssymb} 
\usepackage[pdfencoding=auto]{hyperref}
\usepackage{ mathrsfs }
\allowdisplaybreaks

\newtheorem{theocount}{theocount}
\newtheorem{Thm}[theocount]{Theorem}
\newtheorem{Lemma}[theocount]{Lemma}
\newtheorem{Propo}[theocount]{Proposition}
\newtheorem{Cor}[theocount]{Corollary}
\numberwithin{theocount}{section} 
\newtheorem{Remar}[theocount]{Remark}

\newcommand{\lpi}{\langle}
\newcommand{\rpi}{\rangle}
\newcommand{\bC}{\mathbb{C}}
\newcommand{\bR}{\mathbb{R}}

\newcommand{\bZ}{\mathbb{Z}}
\newcommand{\oz}{{\bar{z}}}
\newcommand{\oze}{{\bar{\zeta}}}
\newcommand{\triv}{\mathrm{triv}}
\DeclareMathOperator{\Cent}{Cent}
\DeclareMathOperator{\Proj}{Proj}
\DeclareMathOperator{\sign}{sign}
\DeclareMathOperator{\End}{End}

\begin{document}
	
	\title{Deformations of unitary Howe dual pairs}
	\author{Dan Ciubotaru}
	\address{Mathematical Institute, University of Oxford, Oxford OX2 6GG, UK}
	\email{dan.ciubotaru@maths.ox.ac.uk}
	\author{Hendrik De Bie}
	\address{Department of Electronics and Information Systems, Ghent University, Krijgslaan 281, 9000 Gent, Belgium}
	\email{Hendrik.DeBie@UGent.be}
	\author{Marcelo De Martino}
	\address{Mathematical Institute, University of Oxford, Oxford OX2 6GG, UK}
	\curraddr{Department of Electronics and Information Systems, Ghent University, Krijgslaan 281, 9000 Gent, Belgium}
	\email{Marcelo.GoncalvesDeMartino@UGent.be}
	\author{Roy Oste}
	\address{Department of Applied Mathematics, Computer Science and Statistics, Ghent University, Krijgslaan 281-S9, 9000 Gent, Belgium}
	\email{Roy.Oste@UGent.be}
	
	\begin{abstract}
		We study deformations of the Howe dual pairs $(\mathrm{U}(n),\mathfrak{u}(1,1))$ and $(\mathrm{U}(n),\mathfrak{u}(2|1))$ to the context of a rational Cherednik algebra $H_{1,c}(G,E)$ associated with a real reflection group $G$ acting on a real vector space $E$ of even dimension. 
		For each pair, we show that the Lie (super)algebra structure of one partner is preserved under the deformation, which leads to a multiplicity-free decomposition of the standard module or its tensor product with a spinor space. 
		For the case where $E$ is two-dimensional and $G$ is a dihedral group, we provide complete descriptions for the deformed pair and the relevant joint-decomposition.
	\end{abstract}
	
	\maketitle

	\section{Introduction}
	
	The topic of Howe dual pairs has been a very fruitful area of research with origin in invariant theory \cite{Ho1} and that has found several applications to automorphic forms \cite{G, Ho2} and the representation theories of real, complex and $p$-adic reductive groups \cite{Ho2, Ho3, Ku, MVW, W}. In particular, it has been an important tool in the study of unitary representations of reductive groups in such pairs \cite{ABPTV, Li}. 
	
	In broad terms, Howe duality relates the representations of a pair $(\mathcal{G},\mathcal{G}')$ of mutually centralising subgroups of (the double cover of) a symplectic group. 
	For the purpose of the present work, we are concerned with the 
	operator commutant version of Howe duality, see also~\cite[Section 2.3]{Ho4}. Hereto, consider the case of a classical complex Lie group  $\mathcal{G}$ and let $\mathfrak{g}'=\textup{Lie}(\mathcal{G}')$ be the subalgebra of the Weyl algebra that generates the algebra of $\mathcal{G}$-invariants. 
	The Weyl algebra acts on the space of polynomial functions on a complex vector space given as a finite direct sum of copies of the fundamental representation of $\mathcal{G}$ or its dual (or more generally, the appropriate Weyl-Clifford algebra and the space of polynomial-spinors, as a combination of symmetric and skew-symmetric variables). 
	The relationship between the representation theory of each member of the pair $(\mathcal{G},\mathfrak{g}')$ is manifest in terms of the multiplicity-free joint-decomposition of the relevant space of polynomials (or polynomial-spinors).
	
	In the present paper, we are interested in studying the deformations of Howe dual pairs by replacing the relevant Weyl algebra and its action on the polynomial module via partial differential operators by the action via Dunkl operators inside a rational Cherednik algebra (in the $t=1$ case) $H_c(G,E)$. 
	The latter algebra is associated with the pair $(G,E)$, where $G$ is a real reflection group acting on the Euclidean space $E 
	$, and to $c$, a $G$-conjugation invariant parameter map on the set of reflections of $G$. 
	Through its action on $E$, the group $G$ can be identified with a finite subgroup of  a classical Lie group  $\mathcal{G}$. 
	The algebras $H_c(G,E)$ arise as deformations of the algebra $\mathcal{W}(E)\rtimes G$, where $\mathcal{W}(E)$ is the Weyl algebra on $E$, in the sense that when $c=0$ we obtain $H_0(G,E) = \mathcal{W}(E)\rtimes G$. 
	
	In general, among the challenges faced when considering a Howe dual pair $(\mathcal{G},\mathfrak{g}')$ in the rational Cherednik context is the fact that the space of Dunkl-operators is not invariant under the action of the classical group $\mathcal{G}$, but rather it is invariant under the finite subgroup $G$ of $\mathcal{G}$ that is part of the data defining $H_c(G,E)$. 
	The natural replacement for the dual partner $\mathcal{G}$ in this context is the algebra that centralizes the Cherednik version of the algebra $\mathfrak{g}'$. 
	Given a classical dual pair $(\mathcal{G},\mathfrak{g}')$, let $\mathcal{A}$ denote the algebra obtained by replacing the relevant Weyl algebra by a rational Cherednik algebra $H_c(G,E)$, i.e., $\mathcal{A}= H_c(G,E)$ or  $\mathcal{A} = H_c(G,E) \otimes \mathcal{C}$ where $\mathcal{C}$ is an appropriate Clifford algebra.   
	The main steps to be considered in the rational Cherednik context are:
	\begin{enumerate}
		\item Find a family $\{\mathfrak{g}'_c\}$ of realisations of the algebra $\mathfrak{g}'$
		as subalgebras of $\mathcal{A} = \mathcal{A}_c$, varying with the parameters $c$.
		\item Determine the subalgebra $\Cent_\mathcal{A}(\mathfrak{g}'_c)$ of $\mathcal{A}$ that centralizes $\mathfrak{g}'_c$.
		\item Describe the joint-decomposition $(\Cent_\mathcal{A}(\mathfrak{g}'_c),\mathfrak{g}'_c)$ of the relevant space of polynomials or polynomial-spinors, for suitable choices of the parameters.
	\end{enumerate}
	
	A satisfactory answer to (1) is not possible for any given Howe dual pair $(\mathcal{G},\mathfrak{g}')$, but examples of favorable situations were considered before in \cite{BSO, CDM, DOV, DOV2, OSS}. 
	The present work is concerned with more examples of such dual pairs. 
	Specifically, from the general scheme for Howe dual pairs over the complex numbers, introduced and described in \cite{Ho1}, one knows the existence of the pairs $(\mathrm{GL}(n,\bC),\mathfrak{gl}(2,\bC))$ and $(\mathrm{GL}(n,\bC),\mathfrak{gl}(2|1,\bC))$.
	We refer also to \cite[Section 5.2]{CW} for a more explicit account on the second pair in the context of Lie superalgebras.
	
	Versions for the theory of dual pairs  for general local fields (outside characteristic $2$) other than the complex numbers were explored in \cite{Ho2,Ho3}. We find the real-form $(\mathrm{U}(n),\mathfrak{u}(1,1))$ of the pair $(\mathrm{GL}(n,\bC),\mathfrak{gl}(2,\bC))$ in the family of dual pairs listed in \cite[(5.2)(iii)]{Ho3}.  Note that a skew duality  for the same pair also appears in the context of Hodge decomposition theory, see~\cite[Example 4(b)]{Ho1}. 
	In the context of a Weyl-Clifford algebra, the real form $(\mathrm{U}(n),\mathfrak{u}(2|1))$ of the dual pair $(\mathrm{GL}(n,\bC),\mathfrak{gl}(2|1,\bC))$  was described in detail in \cite{BDES}, together with the decomposition of the space of polynomial-spinors for the joint action of this pair. For real forms of Lie superalgebras we refer to~\cite[Section 5.3]{Ka}, and \cite{Pa}.

	In the present work, we study the deformations of (the symmetric variables of) the pairs $(\mathrm{U}(n),\mathfrak{u}(1,1))$ and
	$(\mathrm{U}(n),\mathfrak{u}(2|1))$ to a rational Cherednik context. 
	More specifically, for these pairs, we show  in Theorems \ref{u11} and \ref{u21}, below, that (1) is fulfilled. We give partial answers to the desiderata
	(2) and (3) for general reflection groups $G$ in Sections~\ref{Section:scalar} and \ref{Section:Clifford}. 
	One of the main obstacles in finding the concrete centralizers and decompositions in (2) and (3) is the explicit calculation of the eigenvalues of the central element of $\mathfrak{u}(1,1)$; this turns out to be a difficult problem in general. 
	However, in the case when $G$ is a dihedral group acting on a two-dimensional Euclidean space, we give complete answers in Sections~\ref{Section:dihedral} and \ref{Section:CliffordDihedral}.
	
	Another way to view the pair $(\mathrm{U}(n),\mathfrak{u}(1,1))$ is as part of a seesaw dual reductive pair~\cite{Ku2} with a real form of the pair $(\mathrm{O}(2n,\bC),\mathfrak{sl}(2,\bC))$.
	In the differential operator realisation, when acting on the space of polynomials, the lowest weight vectors for the $\mathfrak{sl}(2)$-action are the (spherical) harmonics. 
	For a fixed homogeneous degree, the harmonics form an irreducible representation for the orthogonal group. 
	In the non-deformed case, the pair $(\mathrm{U}(n),\mathfrak{u}(1,1))$ gives a finer decomposition of the harmonics into degrees homogeneous both for a set of complex variables and for the set of complex conjugate variables. 
	This refinement was used in one of Koornwinder's proofs for the addition formula for Jacobi polynomials~\cite{Ko}. 
	In the deformed setting, we show that the refinement still holds, but the interpretation in terms of degrees for complex variables and their conjugates is lost. 
	The deformed harmonics for the dihedral case are described in Theorem~\ref{t:eigenvaluesdihedral}, and in Remark~\ref{r:harmDunkl} we discuss the relation with the harmonics determined by Dunkl~\cite{Du} for the standard polynomial case.

	\section{Preliminaries}
	
	Throughout, the standard notations $[\cdot,\cdot]$ for commutator, and $\{\cdot,\cdot\}$ for anticommutator will be used.
	Consider a real Euclidean vector space $E \cong \mathbb{R}^{N}$ with a positive definite symmetric bilinear form $B$. Denote the natural pairing between $E$ and its dual space $E^*$ by $\langle \cdot,\cdot \rangle \colon E^* \times E \to \mathbb{R}$. The $\bC$-bilinear extension $E^*_\bC\times E_\bC\to\bC$ will also be denoted by
	$\langle \cdot,\cdot \rangle$.
	
	For $G$ a finite reflection group acting on $E$,  
	we denote by $\mathcal S$ the set of reflections of $G$. 
	The action of $G$ on $E$ is naturally extended to $E^*$ as the contragradient action and in turn also to $E_\bC$ and $E^*_\bC$, and to the tensor algebra $T(E_\bC\oplus E^*_\bC)$. We denote by $g \cdot X$ the action of $g\in G$ on $X \in T(E_\bC\oplus E^*_\bC)$, and denote by  $T(E_\bC\oplus E^*_\bC) \rtimes G$  the algebra where the relation $g X = (g \cdot X) g$ holds.

	Given $s\in \mathcal{S}$, fix $\alpha_s$ (resp.~$a_s$) to be a $(-1)$-eigenvector for $s$ acting on $E$ (resp. $E^*$). 
	Denote $\alpha_s^{\vee} = 2 a_s/ \langle  a_s,\alpha_s \rangle \in E^*$. 
	Let $c$ be a map from $\mathcal S$ to $\bC$ (or more generally to a family of variables) 
	such that the elements of an orbit, under the conjugation action of $G$, have the same image. 
	
	The rational Cherednik algebra $H_c(G,E)$ is the quotient of $T(E_\bC\oplus E^*_\bC) \rtimes G$ by the relations
	\begin{equation}
		[\xi,\eta] = 0 \text{ for } \xi,\eta \in E, \quad [x,y] = 0 \text{ for } x,y \in E^*;
		\label{comm}
	\end{equation} 
	\begin{equation}\label{RC}
		[\xi,x] = \langle x,\xi \rangle - \sum_{s\in\mathcal S} \langle  x,\alpha_s\rangle\langle  \alpha_s^{\vee}, \xi \rangle  c(s) s.
	\end{equation} 
		As a vector space, this algebra satisfies the PBW property and its triangular decomposition leads to the notion of category $\mathcal{O}$, see \cite{GGOR} for the details.
		
		\subsection{Complexification and involutions}\label{s:involutions}
		The Euclidean structure $B$ on $E$ gives rise to two non-degenerate pairings on $E_\bC$, a bilinear and a Hermitian one. We denote by $B_\bC:E_\bC\times E_\bC\to\bC$ the bilinear extension of $B$ and by $(\cdot,\cdot)$ the Hermitian one (which we agree to be linear on the first variable). We thus get two identifications $\top,\ast :E_\bC\to E^*_\bC$ defined via
		\[
		\langle \xi^*,\eta\rangle = (\eta,\xi),\qquad\qquad \langle \xi^\top,\eta\rangle = B_\bC(\eta,\xi),
		\]
		for all $\xi,\eta\in E_\bC$. Note that $\top$ is a linear isomorphism while $\ast$ is  conjugate-linear. We let also $\theta:E_\bC\to E_\bC$ be the complex-conjugation automorphism that maps $\xi\otimes\lambda\in E_\bC = E\otimes_\bR\bC$ to $\theta(\xi\otimes\lambda)=\xi\otimes\overline{\lambda}$. We naturally extend $\top,\ast$ and $\theta$ to involutions on $E_\bC\oplus E^*_\bC$, from which we get $\ast = \theta\circ\top = \top\circ\theta$. Note also that
		\begin{equation*}
			(\xi,\eta) = B_\bC(\xi,\theta(\eta)), 
		\end{equation*}
		for all $\xi,\eta\in E_\bC$. We may extend these involutions to involutions or anti-involutions on any associative algebra generated by $E_\bC$.

		\subsection{Standard modules and unitarity}\label{s:unitary} For an irreducible $\mathbb{C}[G]$-module $\tau$ with space $V(\tau)$, we denote by $M_c(\tau) = H_c(G,E) \otimes_{G\ltimes \mathbb{C}[E]} V(\tau) = \mathbb{C}[E] \otimes V(\tau)$ the standard module of $H_c(G,E)$ where $E$ acts by zero on $V(\tau)$.
		When $\tau$ is the trivial module, $M_c(\tau)$ corresponds to the space of polynomials $\mathbb{C}[E]= \mathbb{C}[x_1,\dotsc,x_N]$ (where $x_1,\dotsc,x_{N}$ is a basis of $E^*$).
		In this case, an explicit (faithful~\cite{EtGi}) realisation of $H_c(G,E)$ is given by means of Dunkl operators~\cite{Du} (for the elements of $E$) 
		and coordinate variables (for $E^*$). 
		
		Define $\mathbb{C}_{k}$ to be the space of complex-valued polynomials that are homogeneous of degree $k$ and put
		$M_c(\tau)_k = \bC_k\otimes V(\tau)$.  We shall refer to the subspace $M_c(\tau)_k$ as the degree-$k$ part of $M_c(\tau)$.
		
		Recall that the involution $\ast$ of $E_\bC\oplus E^*_\bC$ described in the previous section yields natural isomorphisms between the polynomial modules $\mathbb{C}[E]=M_c(\textup{triv})$ and $\mathbb{C}[E^*]$ and, moreover, it is easy to show that, for all reflections $s$ of $G$,
		we have $\ast^{-1} \circ s \circ \ast = s$ inside $\End(M_c(\textup{triv}))$. We thus extend the involution $\ast$ of $E_\bC\oplus E_\bC^*$ to a conjugate-linear anti-involution on $H_c(G,E)$, by letting $s^*=s$. Note that if
		$\{\xi_1,\ldots,\xi_N\}\subset E$ and $\{x_1,\ldots,x_N\}\subset E^*$ are orthonormal bases that are in duality, we have $\xi_j^* = x_j$ and $x_j^* = \xi_j$, for all $1\leq j \leq N$.
		Similarly, we extend $\theta$ to a conjugate-linear involution of $H_c(G,E)$ by letting $\theta(s) = s$.

		Consider also the contravariant Hermitian structure $\beta_{c,\tau}(\cdot,\cdot)$ on $M_c(\tau)$ that restricts to a $G$-invariant inner product on $M_c(\tau)_0=V(\tau)$ (see \cite[Section 2.4]{DO}). This form satisfies the contravariance property $\beta_{c,\tau}(h(p),q) = \beta_{c,\tau}(p,h^*(q))$, for all $h\in H_c(G,E)$ and $p,q\in M_c(\tau)$. For real-valued parameters $c$ sufficiently close to $c=0$ (see \cite[Lemma 2.2]{CDM} and the references therein), the Hermitian structure $\beta_{c,\tau}(\cdot,\cdot)$ is in fact positive-definite and $M_c(\tau)$ is a $\ast$-unitarisable $H_c(G,E)$-module.
		This unitary structure on the standard module will be crucial in the decomposition results later on in the paper.

		\subsection{A realisation of \texorpdfstring{$\mathfrak{su}(1,1)$}{su(1,1)}}
		
		For completeness, we first state the following properties.
		\begin{Lemma}\label{Lemma1.1}
			Let $\xi_1,\dotsc,\xi_N$ be an orthonormal basis for $E$ and $x_1,\dotsc,x_{N}$
			a dual basis of $E^*$ such that $\langle x_j, \xi_k \rangle = \delta_{jk}$. In $H_c(G,E)$, the following relations hold
			\begin{equation}
				\bigg[\frac12 \sum_{k=1}^N \xi_k^2, x_j\bigg]=\xi_j, \qquad \bigg[\frac12 \sum_{k=1}^N x_k^2, \xi_j\bigg] = -x_j \rlap{\,.} \label{xp}
			\end{equation} 	
			As a consequence of the preceding relations and the commutativity~\eqref{comm}, we also have	
			\begin{equation}
				[\xi_j,x_k]
				=[\xi_k,x_j]\rlap{\,,}
				\label{ijji}
			\end{equation} 	
			and for all $\xi \in E^*$ and $x\in E$:
			\begin{equation}
				\bigg[\frac12 \sum_{k=1}^N \{\xi_k, x_k\} , \xi\bigg] = -\xi \rlap{\,,}\qquad 	\bigg[\frac12 \sum_{k=1}^N \{\xi_k, x_k\} , x\bigg] = x \rlap{\,.}
				\label{Euler}
			\end{equation} 
		\end{Lemma}
		\begin{proof}
			For relations~\eqref{xp} see for instance \cite{He}. Using~\eqref{xp} and~\eqref{comm}, we find
			\begin{equation*}
				[\xi_k,x_j]=\bigg[\bigg[\frac12 \sum_{l=1}^N \xi_l^2, x_k\bigg],x_j\bigg]
				=\bigg[\bigg[\frac12 \sum_{l=1}^N \xi_l^2, x_j\bigg],x_k\bigg]
				=[\xi_j,x_k] \rlap{\,,}
			\end{equation*} 
			and
			\[
			2 \xi_j =	 \sum_{k=1}^N\big[ \xi_k^2, x_j\big]
			= \sum_{k=1}^N \{\xi_k,[ \xi_k, x_j]\}
			= \sum_{k=1}^N \{\xi_k,[ \xi_j, x_k]\}
			= \sum_{k=1}^N [\xi_j,\{\xi_k , x_k\}]  \rlap{\,,}
			\]
			\[
			-2x_j =	 \sum_{i=1}^N\big[ x_k^2, \xi_j\big]
			= \sum_{k=1}^N \{x_k,[ x_k, \xi_j]\}
			= \sum_{k=1}^N \{x_k,[ x_j, \xi_k]\}
			= \sum_{k=1}^N [x_j,\{x_k , \xi_k\}] \rlap{\,.} \qedhere
			\]
		\end{proof}
		
		\begin{Propo}\label{sl2}
			In $H_c(G,E)$ we have a realisation of the real form $\mathfrak{su}(1,1)$ of the Lie algebra $\mathfrak{sl}(2,\bC)$, generated by
			\begin{align*}
				E_+ &    = \frac12\sum_{j=1}^{N} x_j^2, &
				E_- &    = -\frac12\sum_{j=1}^{N} \xi_j^2 ,
				&
				H &=\sum_{j=1}^{N} x_j\xi_j + \frac{N}{2} -\sum_{s\in\mathcal S} c(s) s  ,
			\end{align*}
			with the following commutation relations
			\begin{align*}
				[H, E_{\pm}] &= \pm 2 E_{\pm}	,
				& [E_+ ,E_- ] &= H. 
			\end{align*}
			In terms of the $\ast$-operation of Section~\ref{s:unitary}, we have $E_{\pm}^* = -E_{\mp}$ and $H^*=H$.
		\end{Propo}
		\begin{proof}
			The realisation of $\mathfrak{su}(1,1)\cong \mathfrak{sl}(2,\bR)$ is well-known (see \cite{He})
			and follows from Lemma~\ref{Lemma1.1}, since by means of~\eqref{RC}, we obtain the expression for $H$ as follows
			\begin{equation}\label{H}
				\frac12 \sum_{j=1}^N \{\xi_j, x_j\} 
				= \frac12\sum_{j=1}^{N} (2x_j\xi_j +   [\xi_j ,x_j ] )
				= \sum_{j=1}^{N} x_j\xi_j + \frac{N}{2} - \sum_{s\in\mathcal S} c(s) s.
			\end{equation}
			
			The identities  $E_{\pm}^* = -E_{\mp}$ and $H^*=H$ are straight-forward, using $\xi_j^* = x_j,x_j^* = \xi_j$, for all $1\leq j \leq N$ and $s^*=s$ for all reflections of $G$.
		\end{proof}
		
		Whenever the parameter map $c$ is real-valued and close enough to $c=0$ (as recalled in Section \ref{s:unitary}), because of the unitarity, the standard module $M_c(\tau)$ decomposes as a direct sum of irreducible representations of $\mathfrak{su}(1,1)$. Let $N_c(\tau)$ denote the scalar by which the central element 
		\begin{equation}\label{e:sigma}
			\sigma = \sum_s c(s)s
		\end{equation} 
		of $\mathbb{C}[G]$ acts on $V(\tau)$. Note that $H$ acts on the degree-$k$ part $M_c(\tau)_k$ by the scalar $(k+\tfrac{N}{2}-N_c(\tau))$ times the identity. We let $\mathcal{H}_c(\tau)_{k}\subseteq M_c(\tau)_k$ be the intersection of $M_c(\tau)_k$ with the kernel of $E_-$. Let $\mathcal{U}\subset H_c(G,E)$ denote the centralizer algebra of the $\mathfrak{su}(1,1)$-subalgebra introduced above. It is known (see \cite{CDM}) that this algebra is the associative subalgebra of $H_c(G,E)$, generated by $G$ and the elements
		\begin{equation}\label{e:X}
			X_{xy} = xy^* - yx^* ,
		\end{equation} 
		with $x,y\in E^*$. 
		A precise description of the relations of this non-homogeneous quadratic algebra was obtained in~\cite{FH}, and in~\cite{DOV} for a more general framework.
		
		For real-valued parameters $c$ as above, it is known that we have a 
		$(\mathcal{U},\mathfrak{su}(1,1))$-decomposition (see \cite{CDM} and also \cite{BSO})  of the standard module 
		\begin{equation}\label{e:sl2decomp}
			M_c(\tau) = \bigoplus_{k\in\mathbb{Z}_{\geq 0}} \mathcal{H}_c(\tau)_{k}\otimes L_{\mathfrak{su}(1,1)}(k + \tfrac{N}{2} - N_c(\tau)),
		\end{equation}
		where $L_{\mathfrak{su}(1,1)}(\mu)$ denotes the infinite-dimensional irreducible unitary lowest-weight module for $\mathfrak{su}(1,1)$ with
		lowest weight $\mu\in\bR$.

		\subsection{Complex structure}\label{s:ComplexStructure} 
		We will give a brief overview of the concept of a complex structure and the notations used in this paper. For a more detailed exposition of this complex structure in the non-deformed setting, see for instance~\cite{BDES} and references therein.  
		See also the related discussion on the Hodge decomposition theory in~\cite[Example 4(b)]{Ho1}. 
		
		Let $E$ be even-dimensional, with $\dim_{\mathbb{R}}(E) =2n$. A complex structure compatible with the Euclidean structure on $E$ is an element $J \in \mathrm{SO}(E,B) \cong\mathrm{SO}(2n,\mathbb{R})$ satisfying $J^T = -J$, hence $J^2 = -1_E$.
		For a given $J$, we can choose a basis for $E$ orthonormal with respect to $B(\cdot,\cdot)$, such that the action of $J$ on $E$ is given by
		\[
		J = \begin{pmatrix} \mathit{0}_n & I_n \\ -I_n & \mathit{0}_n \end{pmatrix}.
		\] 
		We denote this basis by $\{\xi_j\}_{j=1}^{2n}$, and moreover let $\eta_j := \xi_{n+j}$ for $j\in \{1,\dotsc,n\}$.  
		A dual basis of coordinate functions will be denoted by $\{x_j\}_{j=1}^{2n}$, with $y_j := x_{n+j}$ for $j\in \{1,\dotsc,n\}$.  
		We have
		\[
		J \cdot\xi_j = -\eta_j, \qquad J \cdot\eta_j = \xi_j, \qquad 
		J \cdot x_j = -y_j, \qquad J \cdot y_j = x_j\rlap{\,.}
		\]
		
		\begin{Remar}
			A complex structure $J$ does not in general form an automorphism of the rational Cherednik algebra $H_c(G,E)$, compatible with the commutation relations~\eqref{RC}. This is only the case when, on the one hand, the group $G$, viewed, through its action on $E$, as a subgroup of the orthogonal group $\mathrm{O}(E,B)$, is preserved under conjugation with $J$. (Alternatively, this is the case when $J$ preserves the root system associated with $G$.) On the other hand, the images of the parameter map $c$ have to be equal when two orbits are related via conjugation with $J$. 
			
			For the results in the current paper, this compatibility of $J$ and $H_c(G,E)$ is not required. When doing explicit computations, as in Sections~\ref{Section:dihedral} and \ref{Section:CliffordDihedral}, it is convenient to choose $J$ and the action of $G$ on $E$ such that their interaction is somewhat nice.
		\end{Remar}
		
		Via the complex structure $J$, one can make two identifications of $E\cong \mathbb{R}^{2n}$ with $\mathbb{C}^n$ by using the first $n$ coordinates of $E$ as real parts and the last $n$ as imaginary parts. Let $i$ denote the imaginary unit. For $j=1,\dotsc,n$, we define $z_j := x_j + i \, y_{j}$ and $\bar z_j  := x_j - i \, y_{j}$. 
		Similarly, define $\zeta_j := \frac12(\xi_j - i \, \eta_{j})$ and $\bar \zeta_j  := \frac12(\xi_j + i \, \eta_{j})$.
		Classically, when $c = 0$, $\zeta_j $ and $\bar \zeta_j$ correspond to the Wirtinger derivatives $\partial_{z_j} = (\partial_{x_j} - i\partial_{y_j})/2$ and $\partial_{\bar z_j} = (\partial_{x_j} +i \partial_{y_j})/2$. Note that the complex-conjugation involution $\theta$ of $E^*_\bC\oplus E_\bC$ (see \ref{s:involutions}) satisfies $\theta(z_j) = \bar z_j,\theta(\zeta_j) =  \bar \zeta_j$, for all $1\leq j \leq n$. For the involutions we have $(z_j)^\top = 2 \bar \zeta_j$ and $(z_j)^* = 2 \zeta_j$, for all $1\leq j \leq n$.
		
		The complex variables are eigenvectors for the action of the complexification $J_{\bC}$:
		\[
		J_{\bC} \cdot \zeta_j = -i\,\zeta_j, \qquad J_{\bC} \cdot \bar \zeta_j = i\,\bar \zeta_j\rlap{\,,}\qquad J_{\bC} \cdot z_j = i\,z_j, \qquad J_{\bC} \cdot\bar z_j = -i\,\bar z_j\rlap{\,.}
		\]
		Denote by $W^\pm\subset E_\bC$ the $(\mp i)$-eigenspace of $J_\bC$.
		By means of the projection operators $\pi^{\pm} = (1\pm i J_{\mathbb{C}} )/2$, we have $W^{\pm} = \pi^{\pm}(E_{\mathbb{C}})$. 
		These spaces give a direct sum decomposition of  $E_{\mathbb{C}} = W^+ \oplus W^-$ into isotropic spaces with respect to $B_{\bC}$.  Each of the subspaces $W^{\pm}$ is isomorphic to $E$ and to $\mathbb{C}^{n}$, and they are mapped to each other by complex conjugation. 
		There are two Hermitian metrics $h^\pm$ on $E$ given by 
		\[
		h^\pm(\cdot,\cdot) = \frac{1}{2}B(\cdot,\cdot) \pm \frac{i}{2} B(J\cdot,\cdot).
		\] 
		They are such that the identifications of $E$ with $W^\pm$ via $\pi^\pm$ are isometric in the sense that $(\pi^\pm(\xi),\pi^\pm(\eta)) = B_\bC(\pi^\pm(\xi),\pi^\mp(\eta)) = h^\pm(\xi,\eta)$, for all $\xi,\eta\in E$.

		
		There is an action of the unitary group $\mathrm{U}(n)$ that preserves the Hermitian metrics $h^\pm(\cdot,\cdot)$ and each of the spaces $W^{\pm}$. It is the subgroup of $\mathrm{GL}(W^+)$ (or $\mathrm{GL}(W^-)$) commuting with complex conjugation, see~\cite[Example 4(b)]{Ho1}. 
		Explicitly, the unitary group $\mathrm{U}(n)$ is realised inside $\mathrm{O}(2n,\mathbb{R})$ by matrices of the form 
		\[
		A = \begin{pmatrix} B & -C \\ 
			C & B \end{pmatrix} \qquad \text{with }B^TB + C^TC = I_n \text{ and }B^TC - C^TB = \mathit{0}_n
		\]
		where $B,C$ are real $n\times n$ matrices.  
		These are precisely the orthogonal matrices that commute with the complex structure $J$. There are two isomorphisms with $\mathrm{U}(n)$ as $n\times n$ matrix group, given by
		\[
		A = \begin{pmatrix} B & -C \\ 
			C & B \end{pmatrix}  \leftrightarrow B + i C \in \mathrm{U}(n), \quad\text{and}\quad A = \begin{pmatrix} B & -C \\ 
			C & B \end{pmatrix}  \leftrightarrow B - i C \in \mathrm{U}(n),
		\]
		corresponding to the actions of $A$ on the coordinates $z_1,\dotsc,z_{n}$ and $\bar z_1,\dotsc,\bar z_n$ respectively. Hence, if $A\in \mathrm{U}(n)$ acts on $(z_1,\dotsc,z_{n})^T$ through matrix multiplication, its action on $(\bar z_1,\dotsc,\bar z_n)^T$ is given through multiplication by the conjugate transpose or Hermitian conjugate  $A^*$.


		\section{Scalar unitary duality}\label{Section:scalar}
		
		\subsection{A realisation of \texorpdfstring{$\mathfrak{u}(1,1)$}{u(1,1)}}
		
		Recall that $E$ is assumed to be even-dimensional, with $\dim_{\mathbb{R}}(E) =2n$. 
		We first consider the following lemma on the adjoint action of $\mathfrak{su}(1,1)$, generated by $H$ and $E_{\pm}$, on $H_c(G,E)$.  
		
		\begin{Lemma}\label{Lemma1.4}
			For $j \in \{1,\dotsc,n\}$, we have
			\begin{equation}\label{Hzz}
				\begin{aligned}[]
					[E_+, 2\zeta_j]= -\bar z_j,\quad 
					[E_-, z_j]= -2\bar \zeta_j,\quad
					[H,z_j]= z_j,\quad 
					[H, \zeta_j]= -\zeta_j,
					\\
					[E_+,2\bar \zeta_j]= -z_j ,\quad 
					[E_-,\bar z_j]=  -2\zeta_j, \quad 
					[H,\bar z_j]= \bar z_j,\quad 
					[H,\bar \zeta_j]= -\bar \zeta_j.
				\end{aligned}
			\end{equation}
			
		\end{Lemma}
		\begin{proof} 
			These follow immediately from relations~\eqref{xp} and \eqref{Euler}.
		\end{proof}
		
		\begin{Thm}\label{u11}
			In $H_c(G,E)$ we have a realisation of the real form $\mathfrak{u}(1,1)$ of the Lie algebra $\mathfrak{gl}(2,\bC)$, generated by
			\begin{align*}
				E_+ & = \frac12\sum_{j=1}^{n}  z_j \bar z_j   = \frac12\sum_{j=1}^{n} (x_j^2+y_j^2), \\
				E_- &   = -2\sum_{j=1}^{n}  \zeta_j \bar \zeta_j   = -\frac12\sum_{j=1}^{n} (\xi_j^2+\eta_j^2) ,
				\\
				H &=  \sum_{j=1}^{n}( z_j\zeta_j  + \bar z_j\bar \zeta_j ) + n -\sum_{s\in\mathcal S} c(s) s =\sum_{j=1}^{n} ( x_j\xi_j  +  y_j\eta_j)+ n -\sum_{s\in\mathcal S} c(s) s 
				\\
				Z_0  &= \sum_{j=1}^{n} (\bar z_j \bar \zeta_j - z_j \zeta_j ) = i \sum_{j=1}^{n} (x_j\eta_j-y_{j}\xi_j)  ,
			\end{align*}
			where $Z_0$ is central, and we have the following commutation relations
			\begin{align*}
				[H, E_{\pm}] &= \pm 2 E_{\pm}	,
				& [E_+ ,E_- ] &= H. 
			\end{align*}
			The $\mathfrak{su}(1,1)$ subalgebra is generated by $H,E_+,E_-$.
			In terms of the $\ast$-operation of Section~\ref{s:unitary}, we have $E_{\pm}^* = -E_{\mp}$, $H^*=H$ and $Z_0^* = Z_0$.
		\end{Thm}
		
		\begin{proof}
			See Proposition~\ref{sl2} and the definitions in the previous section. 
			The fact that $Z_0$ is central follows by means of 
			Lemma~\ref{Lemma1.4}.
		\end{proof}
		
		\begin{Remar}
			One can easily show that the $\mathfrak{su}(1,1)$ algebra generated by $H,E_\pm$ is independent of the choice of orthonormal basis made. On the other hand, the central element $Z_0$ is only invariant under basis transformation by elements of $U(n)$, as described in Section~\ref{s:ComplexStructure}. 
		\end{Remar}
		
		Classically, when $c = 0$, the linear combinations 
		\[H_z = (H-Z_0)/2,\quad\text{and}\quad H_{\bar z } = (H+Z_0)/2
		\] 
		measure, respectively, the holomorphic and anti-holomorphic degree of a monomial, so that $Z_0$ induces an operator on the polynomial module that measures the difference between the holomorphic and anti-holomorphic degrees.
		
		\begin{Remar}\label{Remark:Z0}
			There is some freedom involved in the definition of the central element of $\mathfrak{u}(1,1)$ as for every $\epsilon \in \mathbb{R}$, the linear combination $\tilde{Z}_0 = Z_0+\epsilon\sigma$ (with $\sigma$ as in (\ref{e:sigma})) also commutes with the triple $\{H,E_+,E_-\}$, satisfies $\tilde{Z}_0^*=\tilde{Z}_0$, and, when $c=0$, also restricts to the operator that measures the difference between holomorphic and antiholomorphic degrees.
		\end{Remar}

		We conclude this section with some properties we will use later on. 
		
		\begin{Lemma}\label{Lemma1.2} 
			In $H_c(G,E)$ , for $j,k \in \{1,\dotsc,n\}$, 
			\[
			[ \bar \zeta_j, z_k] = 	[ \bar \zeta_k, z_j],\qquad 
			[ \zeta_j, \bar z_k] = 	[ \zeta_k, \bar z_j],\qquad
			[ \zeta_j, z_k] = 	[ \bar \zeta_k, \bar z_j].
			\]
		\end{Lemma}
		\begin{proof} 
			As a consequence of relation~\eqref{ijji} of Lemma~\ref{Lemma1.1}, we have for $j,k \in \{1,\dotsc,n\}$:
			\begin{equation}\label{jkkj}
				[\xi_j,x_k]=[\xi_k,x_j],\quad 
				[\xi_j,y_k]=[\eta_k,x_j],\quad 
				[\eta_j,y_k]=[\eta_k,y_j].
			\end{equation}
			The stated results follow after applying these relations~\eqref{jkkj}
			to
			\begin{align*}
				[ 2\bar \zeta_j, z_k] 	& =	[\xi_j + i \, \eta_j,x_k + i \, y_k]
				=  [\xi_j,x_k ]- [\eta_j,y_k] + i( [\xi_j,y_k] +[\xi_k,y_j ] ),\\
				[ 2\zeta_j, \bar z_k]
				&= [\xi_j - i \, \eta_j,x_k - i \, y_k]
				=  [\xi_j,x_k ]- [\eta_j,y_k] - i( [\xi_j,y_k] +[\xi_k,y_j ] ),
				\\
				[ 2\zeta_j,  z_k]&
				=[\xi_j - i \, \eta_j,x_k + i \, y_k]
				=  [\xi_j,x_k ]+ [\eta_j,y_k] + i( [\xi_j,y_k] -[\xi_k,y_j ] ),
				\\
				[2\bar \zeta_j, \bar z_k] 
				&=[\xi_j + i \, \eta_j,x_k - i \, y_k]
				=  [\xi_j,x_k ]+ [\eta_j,y_k] - i( [\xi_j,y_k] -[\xi_k,y_j ] ).
			\end{align*}
			Note that for $j=k$, the imaginary part of $[\zeta_j,z_j] = [\bar \zeta_j, \bar z_j]$ vanishes. 
		\end{proof}

		\subsection{Scalar decomposition} 
		
		For the $\mathfrak{su}(1,1)$ algebra of Proposition~\ref{sl2}, there is a natural invariance with respect to the finite reflection group $G$. For the $\mathfrak{u}(1,1)$ algebra of Proposition~\ref{u11}, the invariance in question would be under a finite subgroup of the unitary group: $G^J :=G \cap \mathrm{U}(n)$, the subgroup of $G$ consisting of all elements $w\in G$ that commute with the complex structure $J$ in $\End(E)$.
		
		
		\begin{Propo}\label{p:gpinvariance}
			The elements $E_+,E_-,H$ are invariant under the action of the group $G$. 
			The element $Z_0$ is invariant under $G^J$. 
		\end{Propo}
		\begin{proof}
			We use the following notations. Consider the $n\times 1$ column matrices
			\begin{equation}\label{notat}
				z = \begin{pmatrix} z_1  \\ \vdots  \\  z_n\end{pmatrix} , 
				\quad \bar z = \begin{pmatrix} \bar z_1 \\ \vdots  \\  \bar z_n\end{pmatrix},
				\quad \zeta = \begin{pmatrix}  \zeta_1 \\ \vdots  \\  \zeta_n\end{pmatrix},
				\quad x = \begin{pmatrix} x_1 \\ \vdots  \\  x_n\end{pmatrix},
				\quad \xi = \begin{pmatrix} \xi \\ \vdots  \\  \xi_n\end{pmatrix}, 
				\quad \text{etc.}
			\end{equation}
			We have  \[
			E_+  = \frac12 \bar z^T z 
			= 
			\frac12\begin{pmatrix} x  \\ y \end{pmatrix}^T \begin{pmatrix} x  \\ y \end{pmatrix}.
			\]
			For $g \in G$, through its matrix representation $[g]$ wit h respect to our chosen bases of $E$ and $E^*$ we find  
			\[
			g \cdot E_+ 
			= \frac12 \left( [g] \begin{pmatrix} x  \\ y \end{pmatrix}\right)^T [ g] \begin{pmatrix} x  \\ y \end{pmatrix}
			= \frac12 \begin{pmatrix} x  \\ y \end{pmatrix}^T [g]^T [ g] \begin{pmatrix} x  \\ y \end{pmatrix}
			= E_+.
			\]
			In the same way we find that $E_-$ and $H$ are $G$-invariant. However, as
			\[
			Z_0 =   \bar z^T\bar \zeta- z^T \zeta = i \begin{pmatrix} x  \\ y \end{pmatrix}^T J \begin{pmatrix} \xi  \\ \eta \end{pmatrix},
			\]
			this is invariant under the action of $g \in G$ if and only if $g$ commutes with the complex structure $J$, i.e., if $g$ corresponds to an element of $\mathrm{U}(n)$.
		\end{proof}
		
		\begin{Remar}
			It is natural to ask what is the structure of the finite group $G^J$. The case for $G$ a dihedral group, that is, when $n=1$, will be treated in the next section, see Proposition~\ref{p:GJ}. When the complex structure $J$ preserves the root system $R=R(G)$ of $G$, it turns out that $G^J$ has the structure of a complex reflection group of rank $n$. For $n>1$ and $R$ irreducible, this statement is a special case of \cite[Theorem 3.2]{Br}. We list them here for convenience:
			\[
			\begin{aligned}
				G &= G(B_{2n}),n > 1,&\qquad G^J&=G(4,1,n), \\
				G &= G(D_{2n}), n>1,&\qquad G^J&=G(4,2,n), \\
				G &= G(E_8),&\qquad G^J&=G_{31},\\
				G &= G(F_4),&\qquad G^J&=G_{8}\textup{ or }G_{12},\\
				G &= G(H_4),&\qquad G^J&=G_{22}.
			\end{aligned}
			\]
			Note that for
			$F_4$, there are two conjugacy classes of automorphisms of $R(F_4)$ with characteristic polynomial $(x^2 + 1)^2$, so there are
			two inequivalent choices of $J$ preserving that root system.

			For $n>1$ and $R$ reducible, since $J$ is an automorphism of $R$, it induces a permutation on the set of connected components of the Dynkin diagram of $R$. Since moreover $J^2 = -1$, this induced permutation is of order $2$. Hence, $J$ either preserves a connected component of the graph, or it switches two connected components, forcing them to be isomorphic. From this, it follows that $G^J = G_1\times\cdots\times G_p$ in which $G_j$ is a complex reflection group described from the above results on irreducible root systems, if $J$ preserves a component, and it is the diagonal copy of a real reflection subgroup in case $J$ switches two connected components of the graph.
			
		\end{Remar}
		
		We now refine the $\mathfrak{su}(1,1)$ decomposition (\ref{e:sl2decomp}). Since $Z_0$ preserves each $M_c(\tau)_k$ and commutes with the triple $\{ H, E_+, E_-\}$, each of the spaces $\mathcal{H}_c(\tau)_{k}$ is $Z_0$-invariant. Furthermore, for each $k\geq 0$, let $\Sigma_c(\tau)_k$ denote the set
		of eigenvalues of $Z_0$ on $\mathcal{H}_c(\tau)_k$. We have the following:
		
		\begin{Propo}\label{p:ZnaughtProp}
			Assume $c$ is real-valued and sufficiently close to $c=0$ such that $M_c(\tau)$ is irreducible and unitarisable. We have:
			\begin{enumerate}
				\item[(1)] The operator $Z_0$ restricted to $\mathcal{H}_c(\tau)_k$ is diagonalisable for all $k\geq 0$ and $\Sigma_c(\tau)_k\subseteq \bR$.
				\item[(2)] If $\lambda\in\Sigma_c(\tau)_k$ then $-\lambda$ is also in $\Sigma_c(\tau)_k$.
			\end{enumerate}
		\end{Propo}
		
		\begin{proof}
			Item (1) follows from the fact that $Z_0$ acts as a Hermitian operator on $\mathcal{H}_c(\tau)_k$, with respect to the unitary structure on $M_c(\tau)$ described in Section~\ref{s:unitary}. Recall that the assumption on $c$ guarantees that the natural Hermitian pairing on $M_c(\tau)$ is in fact positive-definite.
			For item (2), we denote also by $\theta\in \End_\bR(M_c(\tau))$ the complex conjugation involution described in Section~\ref{s:involutions}, which fixes all $x_j\otimes v\in M_c(\tau)_1$, for $v\in V(\tau)$, and satisfies $\theta(z_j\otimes v_k)=\overline{z}_j\otimes v_k$. One readily checks that $\theta\circ Z_0\circ\theta = -Z_0$.  Thus, if $p\in\mathcal{H}_c(\tau)_k$ is an eigenvector with eigenvalue $\lambda$, we obtain that 
			\[
			Z_0(\theta(p)) = \theta(\theta\circ Z_0 \circ \theta(p)) = -\overline{\lambda}\theta(p).
			\]
			Claim (2) now follows from item (1).
		\end{proof}
		
		Because of Proposition \ref{p:ZnaughtProp}(2), we can write  $\Sigma_c(\tau)_k = \Sigma^+_c(\tau)_k \cup \Sigma^-_c(\tau)_k \cup \Sigma^0_c(\tau)_k$ where $\Sigma^\pm_c(\tau)_k$ denotes the set of positive (respectively, negative) eigenvalues in $\Sigma_c(\tau)_k$ and $\Sigma^0_c(\tau)_k=\{0\}$ or $\emptyset$ depending if the eigenvalue zero is present or not. For $\lambda\in \Sigma_c(\tau)_k$, let $\mathcal{H}_c(\tau)_{k,\lambda}$ denote the $\lambda$-eigenspace of the operator $Z_0$ inside $\mathcal{H}_c(\tau)_k$. We can thus decompose the space of harmonics as
		\[
		\mathcal{H}_c(\tau)_k = \mathcal{H}_c(\tau)_{k,0} \oplus
		\bigoplus_{\lambda\in \Sigma^+_c(\tau)_k}
		(\mathcal{H}_c(\tau)_{k,\lambda}\oplus \mathcal{H}_c(\tau)_{k,-\lambda}).
		\]
		Furthermore, if $\mathcal{V}$ denotes the centralizer algebra of $\mathfrak{u}(1,1)$ inside $H_c(G,E)$, then $\mathcal{V}$ is the subalgebra of $\mathcal{U}=\textup{Cent}_{H_c(G,E)}(\mathfrak{su}(1,1))$ that commutes with the central element of $\mathfrak{u}(1,1)$ and we have $G^J\subset \mathcal{V}$. This brings us to the following refinement of the decomposition~\eqref{e:sl2decomp}.
		
		\begin{Cor}
			For real-valued $c$ such that $M_c(\tau)$ is irreducible and unitarisable, we have
			the following joint $(\mathcal{V},\mathfrak{u}(1,1))$-decomposition of the standard module
			\[
			M_c(\tau) = \bigoplus_{k\in \mathbb{Z}_{\geq 0}}\bigoplus_{\lambda\in\Sigma_c(\tau)_k} 
			\mathcal{H}_c(\tau)_{k,\lambda}\otimes L_{\mathfrak{u}(1,1)}(k + n - N_c(\tau),\lambda),
			\]
			where $L_{\mathfrak{u}(1,1)}(\mu,\nu)$ denotes the irreducible lowest weight $\mathfrak{u}(1,1)$-representation of $(H,Z_0)$-lowest weight $(\mu,\nu)$.
		\end{Cor}
		
		We note that finding explicit formulas for the eigenvalues of $Z_0$, or for the generators of the centralizer algebra $\mathcal{V}$, are non-trivial problems for general reflection group $G$ and dimension $n$. From computations, it is observed that when using a slightly different central element for $\mathfrak{u}(1,1)$, as noted in Remark~\ref{Remark:Z0}, the expressions for its eigenvalues are sometimes simpler. 
		The properties of Proposition~\ref{p:ZnaughtProp} remain valid for a central element of the form  $\tilde{Z}_0 = Z_0+\epsilon\sigma$, though the involution linking the eigenvectors corresponding to two eigenvalues with opposite sign will be more involved than just complex conjugation. 
		Central elements of this form will prove to be useful in Section~\ref{Section:CliffordDihedral}. Next, we consider an example where it is feasible to compute the eigenvalues and centraliser explicitly.

		\section{Example: dihedral groups} \label{Section:dihedral}
		
		\subsection{Preliminaries}

		We consider in particular the case $n=1$, which remains fixed for this section. 
		Within this section, we will use the letter $n$ as an (integer) index or variable without connection to the dimension of $E$ or the number of complex variables. 
		We will use the letter $m$ to refer to the number of reflections of
		a dihedral group, so the order of the group is $2m$. 
		In this case, the group $G^J$ is a cyclic group. 
		
		\begin{Propo}\label{p:GJ}
			If $G = I_2(m)$ is a dihedral group acting on $E=\mathbb{R}^2$ and $J\in O(E)$ is a complex structure, then $G^J = C_m$, the cyclic group of order $m$.
		\end{Propo}
		
		\begin{proof}
			It is a general fact that if $g\in O(E)$ commutes with a compatible complex structure $J\in O(E)$, then it must be that $\det(g)=1$. So, in general, we have $G^J\subset G\cap SO(E)$. Since in the present case $SO(E)$ is abelian and $J\in SO(E)$, it follows that
			$G^J = G\cap SO(E)$. This finishes the proof.
		\end{proof}
		
		We shall determine the scalar decomposition of the modules $M_c(\tau)$ when $E=\bR^2$ and $G=I_2(m)$ is a dihedral group and compute explicitly the eigenvalues of the central element of $\mathfrak{u}(1,1)$ on the space of harmonics, $\mathcal{H}_c(\tau)$. The harmonics of the polynomial module (when $\tau=\triv$) were studied by Dunkl in \cite[Section 3]{Du}. As $E$ is only two-dimensional, we omit the subscripts and use the notation $\xi,\eta\in E$ for an orthonormal basis and $x,y\in E^*$ for its dual basis. The action of $G$ on $E$ is chosen such that, in terms of the orthonormal basis, the root system of $G$ is
		$R(m) = \{\alpha_1,\ldots,\alpha_{2m}\}\subseteq E,$
		where for each $1 \leq j \leq 2m$ we have
		\begin{equation}\label{rootsdihedral}
			\alpha_j := 
			\sin(\pi j/m)\xi - \cos(\pi j/m)\eta.
		\end{equation}
		The positive roots are $\{\alpha_j\mid1\leq j \leq m\}$. Let also $a_j = \sin(\pi j/m)x - \cos(\pi j/m)y$, from which $\langle a_j,\alpha_j \rangle = 1$. Label the reflections as $s_1,s_2,\ldots,s_{m}$ with $s_j = s_{\alpha_j}$. We note that $s_m$ is the reflection that fixes the $x$-axis when acting on $E^*$. 
		When $m$ is even there are two conjugacy classes of reflections, one corresponding to $s_j$ for $j$ even, and one to $s_j$ for $j$ odd. In this case, the parameter function will be denoted $c_{\bar 0} = c(s_j)$ for $j$ even, and $c_{\bar 1} = c(s_j)$ for $j$ odd. 
		When $m$ is odd there is only one conjugacy class of reflections, so the parameter function can be replaced by a single constant $c= c(s_j)$ for all $j$. For a fixed $m$, define $\omega := \exp(2\pi i/m)$ a primitive $m$-th root of $1$ and, for $n\in \bZ$, define the elements
		\[
		\sigma(n) := \sum_{j=1}^{m}c(s_j)\omega^{nj} s_j \in H_c(G,E),
		\]
		and let $\sigma_n = \sum_{j=1}^{m}c(s_j)\omega^{nj}$ denote the scalar by which the element
		$\sigma(n)$ acts on trivial representation of $G$. Note that $\sigma(n) = \sigma(q)$ if $n \equiv q$ modulo $m$, and that we can write 
		$\sigma_n = \sigma_{n,\bar{0}}+\sigma_{n,\bar{1}}\in\bC$ where $\bar u$ is the residue modulo $2$ of $u\in \bZ$ and
		\begin{equation}\label{e:omegas}
			\sigma_{n,\bar{u}} := \sum_{\bar j = \bar u}c(s_j)\omega^{nj}
		\end{equation}
		with the summation running over $j$, from $1$ to $m$, such that $\bar j = \bar u$.
		
		\begin{Lemma}\label{p:sigmas}
			If $m$ is odd, we have $\sigma_n=0$ except if $n \equiv 0 \mod m$, in which case $\sigma_0 = mc$. If $m=2q$ is even, we have $\sigma_n=0$ except if $n \equiv 0$ or $n \equiv q \mod m$, where we have $\sigma_0 = q(c_{\bar 0} + c_{\bar 1})$ and $\sigma_q = q(c_{\bar 0} - c_{\bar 1})$. In particular, $\sigma_n \in \bR$ when $c$ is real-valued.
		\end{Lemma}
		
		\begin{proof}
			When $m$ is odd, the parameter function is constant, and for $n \equiv 0 \mod m$, the result is immediate from $\omega^m=1$. For other values of $n$, the claim follows from the elementary polynomial factorization $X^m-1 = (X-1)(\sum_{j=0}^{m-1}X^{j})$ for $X= \omega^n$. When $m=2q$ is even, we have
			\[
			\sigma_n = \sum_{j=1}^q(c_{\bar 0}(\omega^{n})^{2j} + c_{\bar 1}(\omega^{n})^{2j-1}) = (c_{\bar 0} + c_{\bar 1}\omega^{-n})  \sum_{j=1}^q(\omega^{2n})^{j}.
			\]
			For $n \equiv 0$ or $n \equiv q \mod m$, one has $ \omega^{2n}=1$ so the result follows using $\omega^{-q}=-1$.
			For other values of $n$, we have $\sum_{j=0}^{q-1}\omega^{2nj} =  ((\omega^{2n})^q-1)/(\omega^{2n}-1) =0$.
		\end{proof}
		
		\subsection{Representations}
		
		We now describe the decomposition of $M_c(\tau)$ for the simple modules of $\mathbb{C}[G]$. Recall that such a $\tau$ is obtained from a one- or two-dimensional representation of $G$.
		When $m$ is even, there are four one-dimensional representations of $G$, the trivial, the sign and two other which we shall denote by $\chi_{\bar 0},\chi_{\bar 1}$ and are characterized by $\chi_{\bar u}(s_k) = -1$ if $\bar u = \bar k$ and is $1$ otherwise. Using the notation in (\ref{e:omegas}), it follows that
		\begin{align*}
			\triv(\sigma(n)) &= \sigma_{n,\bar 0} + \sigma_{n,\bar 1} = \sigma_n\\
			\chi_{\bar 0}(\sigma(n)) &= -\sigma_{n,\bar 0} + \sigma_{n,\bar 1} \\
			\chi_{\bar 1}(\sigma(n)) &= \sigma_{n,\bar 0} - \sigma_{n,\bar 1} \\
			\sign(\sigma(n)) &= -\sigma_{n,\bar 0} - \sigma_{n,\bar 1} = -\sigma_n.
		\end{align*}
		When $m$ is odd, there are only two one-dimensional representations, the trivial and the sign. If $\chi$ denotes a one-dimensional representation of $G$, when $n=0$, we denote $\chi(\sigma(0)) = N_c(\chi)$. From Lemma~\ref{p:sigmas}, for real-valued $c$, we have $\chi(\sigma(n))\in \bR$ for  all $n\in \bZ$ and \begin{equation}\label{e:chisigma}	
			\chi(\sigma(n))=\chi(\sigma(-n)).
		\end{equation}

		If we write $m=2q$ when $m$ is even and $m=2q-1$ when it is odd, there are, in both cases, $(q-1)$ two-dimensional representations labelled by $\rho^{(u)}$, with $u=1,\ldots, q-1$ and which can be realized in $E_\bC$ via the formulas
		\[
		\rho^{(u)}(s_j)(z) = \omega^{ju}\oz, \qquad\qquad \rho^{(u)}(s_j)(\oz) = \omega^{-ju}z.
		\]
		Of course, $\rho^{(1)}$ is just the reflection representation.

		For $\tau$ an irreducible representation of $G$ and $k \in \mathbb{Z}$, denote $\sigma(\tau)_k$ the set with, depending on $\tau$, the following elements (using the notation~\eqref{e:omegas})
		\[
		\begin{array}{l|l}
			\tau & \sigma(\tau)_k\\\hline
			\textup{triv} & \{\sigma_{k,\bar 0} + \sigma_{k,\bar 1}\} =\{ \sigma_k\}\\
			\chi_{\bar 0} & \{-\sigma_{k,\bar 0} + \sigma_{k,\bar 1}\} \\
			\chi_{\bar 1} & \{\sigma_{k,\bar 0} - \sigma_{k,\bar 1}\} \\
			\textup{sign} & \{-\sigma_{k,\bar 0} - \sigma_{k,\bar 1}\} =\{ -\sigma_k\} \\
			\rho^{(u)} &  \{ \sigma_{k+u}, \sigma_{k-u} \}
		\end{array}
		\]
		Moreover, in $V(\tau)$, the representation space of $\tau$, for $\sigma_{\tau,k} \in \sigma(\tau)_k$, the notations $z(\sigma_{\tau,k})$ and $\bar z(\sigma_{\tau,k})$ will both refer to the single basis vector of $V(\tau)$ when $\tau$ is one-dimensional. For two-dimensional $\tau = \rho^{(u)}$, with $z,\bar z$ as basis for $V(\tau)$, we set $z(\sigma_{k+u}) := z$ and $\bar z(\sigma_{k+u}):= \bar z$, while $z(\sigma_{k-u}) := \bar z$ and $\bar z(\sigma_{k-u}):= z$.

		\begin{Thm}\label{t:eigenvaluesdihedral}
			Let $\tau$ be an irreducible representation of $G$ and $c$ real-valued and close enough to $c=0$ such that $\pi:H_c(G,E)\to \End(M_c(\tau))$ is irreducible and unitarisable. For $k > 0$, the eigenvalues of $Z_0 + \epsilon \sigma$ on the space of polynomials of homogeneous degree $k$ are
			\begin{equation}\label{e:eigenvaluesdihedral}
				\lambda(\sigma_{\tau,k})^{\pm} = \pm\sqrt{(k- N_c(\tau))^2 -(1-\epsilon^2)\sigma_{\tau,k}^2}\rlap{\,,} \qquad \sigma_{\tau,k} \in \sigma(\tau)_k \rlap{\,.}
			\end{equation}
			The corresponding eigenvectors inside the space of harmonics $\mathcal{H}_c(\tau)$, and hence a basis for $\mathcal{H}_c(\tau)_k$,
			are 
			\begin{equation}\label{e:eigenvectorsdihedral}
				\begin{split}
					h(\epsilon,\sigma_{\tau,k})^{+} &:= \Proj_k\left[ ( k - N_c(\tau) + \lambda(\sigma_{\tau,k})^{+}) \bar z^k \otimes \bar  z(\sigma_{\tau,k}) - \sigma_{\tau,k} (1-\epsilon)  z^k \otimes z(\sigma_{\tau,k})  \right] \\
					h(\epsilon,\sigma_{\tau,k})^{-} &:= \Proj_k\left[ ( k - N_c(\tau) - \lambda(\sigma_{\tau,k})^{-}) z^k \otimes  z(\sigma_{\tau,k}) - \sigma_{\tau,k} (1+\epsilon) \bar z^k \otimes \bar  z(\sigma_{\tau,k}) \right]       
				\end{split}
			\end{equation}
			where the projection onto $\ker E_-$ is given by
			\begin{equation}\label{e:Proj}
				\Proj_k \colon p \mapsto \sum_{j=0}^{\lfloor k/2\rfloor} \frac{(-1)^j }{j! (N_c(\tau) -k+1)_j} E_+^j E_-^j p\,,
			\end{equation}
			using the notation 
			$(a)_j=a(a+1)\cdots(a+j-1)$ for $j\in\mathbb{Z}_{>0}$ and $(a)_0=1$.
			
			When the polynomial degree is 0, for $\tau$ one-dimensional, we have $\epsilon$ times the element of the singleton $\sigma(\tau)_0$ as eigenvalue for $Z_0 + \epsilon \sigma$. For $\tau=\rho^{(u)}$, the eigenvalue is $0$. In both cases, the eigenvectors are given by the representation space $1\otimes V(\tau)$.
			
			All other eigenvectors of $Z_0 + \epsilon \sigma$ are generated through action of $E_+$.
		\end{Thm}
		\begin{proof}
			The proof of Theorem~\ref{t:eigenvaluesdihedral} is the subject of Section~\ref{s:proof}. We will first discuss some remarks and consequences. 
		\end{proof}
		
		\begin{Remar}\label{r:harmDunkl}
			An interesting property of the eigenvectors~\eqref{e:eigenvectorsdihedral} is that for $\epsilon=-1$, the eigenvector $h(-,\sigma_{\tau,k})^{+}$ is in the kernel of $\zeta$, while for $\epsilon=1$, the eigenvector $h(+,\sigma_{\tau,k})^{-}$ is in the kernel of $\bar\zeta$ (this is a consequence of the results in Section~\ref{Section:CliffordDihedral}). For $\tau = \triv$, these are precisely the harmonic polynomials determined by Dunkl in~\cite[Section~3]{Du}. There it was noted that these kernels need not be orthogonal, which is observable immediately from the formulas here, orthogonality of the kernels holding only when $\sigma$ acts by $0$. 
		\end{Remar}
		
		In the classical non-deformed case, with the full $\mathrm{GL}(2)$-action on the Weyl algebra associated with $E=\bR^2$, the centraliser of the Lie subalgebra $\mathfrak{u}(1,1)$ of $\mathrm{U}(1)$-invariants (basically the one from Proposition~\ref{u11} with $c=0$) is generated by $Z_0$, which forms a realisation of the Lie algebra $\mathfrak{u}(1)$. In this low-dimensional case, the structure of the centraliser is preserved under the deformation. 
		
		\begin{Propo}\label{p:cent}
			For $\mathfrak{u}(1,1)$ with central element $\tilde{Z}_0 =Z_0 + \epsilon\sigma$,  
			the centraliser $\Cent_{H_c(G,E)}(\mathfrak{u}(1,1))$ is generated by the central element of $\mathfrak{u}(1,1)$ and $\mathbb{C}[G^J]$.
			
		\end{Propo}
		\begin{proof}
			It was obtained in~\cite{CDM}, that $\mathcal{U}\subset H_c(G,E)$, the centralizer algebra of the $\mathfrak{su}(1,1)$-subalgebra of $H_c(G,E)$, is generated by $G$ and the elements of the form~\eqref{e:X}. 
			For $E=\bR^2$, with basis $\xi,\eta\in E$ and dual basis $x,y\in E^*$, there is only a single non-zero element of this form, namely $x\eta -y\xi = -i Z_0 $. A general element of $\mathcal{U}$ is hence a polynomial in $\tilde{Z}_0 =Z_0 + \epsilon\sigma$ with coefficients in $\mathbb{C}[G]$. In light of Proposition~\ref{p:gpinvariance}, and noting that $\sigma$ is $G$-invariant, the desired result follows. 
		\end{proof}
		
		Using the previous result, we have the following explicit decomposition formulas, where the Lie algebra $\mathfrak{u}(1,1)$ has $Z_0+ \epsilon \sigma$ as central element for $\epsilon \in \mathbb{R}$. Let $\mathcal V = \Cent_{H_c(G,E)}(\mathfrak{u}(1,1))$. 
		
		\begin{Cor}
			Let
			$\chi$ be a one-dimensional irreducible representation of $G$. 
			For $k \in \bZ_{>0}$, denote by $\sigma_{\tau,k}$ the sole element of the set $\sigma(\tau)_k$, and  let $\lambda_0 := \epsilon\sigma_{\tau,0}$ while for $k \in \bZ_{\neq 0}$, 
			$\lambda_k := \lambda(\sigma_{\tau,|k|})^{\sign(k)}$ as defined in~\eqref{e:eigenvaluesdihedral}. 
			Let $\mathcal{H}_c(\chi)_{\lambda_0} = \bC_\chi$ be the representation space of $\chi$ and $\mathcal{H}_c(\chi)_{\lambda}$ denote the $\lambda$-eigenspace of $Z_0+ \epsilon \sigma$ in $\ker E_-$. 
			
			For $c$ as in Theorem~\ref{t:eigenvaluesdihedral}, the joint $(\mathcal V,\mathfrak{u}(1,1))$-decomposition of $M_c(\chi)$ is
			\[
			M_c(\chi) = 
			\bigoplus_{k \in \bZ}
			\mathcal{H}_c(\chi)_{\lambda_k}\otimes L_{\mathfrak{u}(1,1)}(|k|+1-N_c(\chi),\lambda_k)
			.
			\]
			Each of the spaces $\mathcal{H}_c(\chi)_{\lambda_k}$ is a one-dimensional irreducible module for $\mathcal V$.
		\end{Cor}
		
		For the following corollary, note that when $\tau$ is a two-dimensional irreducible representation of $G$, the scalar $N_c(\tau) = 0$.
		
		\begin{Cor}
			Let $\tau=\rho^{(u)}$ be a two-dimensional irreducible representation of $G$, and 
			denote by $\Lambda_k = \{ \lambda(\sigma_{\tau,k})^{\pm} \mid \sigma_{\tau,k} \in \sigma(\tau)_k\}$ the set of eigenvalues of $Z_0+ \epsilon \sigma$ as defined in~\eqref{e:eigenvaluesdihedral}, by $V(\tau)$ the representation space of $\tau$, and by $\mathcal{H}_c(\tau)_{\lambda}$ the $\lambda$-eigenspace of $Z_0+ \epsilon \sigma$ in $\ker E_-$.
			
			For $c$ as in Theorem~\ref{t:eigenvaluesdihedral}, the joint $(\mathcal{V},\mathfrak{u}(1,1))$-decomposition of $M_c(\tau)$ is
			\begin{align*}
				M_c(\tau) = 
				V(\tau)\otimes L_{\mathfrak{u}(1,1)}(1,0)\oplus
				\bigoplus_{k \in \bZ_{>0}}\bigoplus_{\lambda_k \in \Lambda_k} &
				\mathcal{H}_c(\tau)_{\lambda_k}\otimes L_{\mathfrak{u}(1,1)}(k+1,\lambda_k)
			\end{align*}
			The space $\mathcal{H}_c(\chi)_{\lambda_k}$ is a one-dimensional irreducible module for $\mathcal V$ if the set $\Lambda_k$ has four distinct eigenvalues. If $\Lambda_k$ contains only two distinct eigenvalues, $\mathcal{H}_c(\chi)_{\lambda_k}$ is two-dimensional and hence consists of two irreducible modules for $\mathcal V$.
		\end{Cor}
		
		The remainder of this section will be devoted to proving Theorem~\ref{t:eigenvaluesdihedral}.
		
		\subsection{Proof of Theorem~\ref{t:eigenvaluesdihedral}}\label{s:proof}
		
		We recall that $z = x+iy$, $\zeta = \tfrac{1}{2}(\xi - i\eta)$ and that $\oz,\oze$ are defined by complex-conjugation.
		
		\begin{Lemma}\label{l:dihprelim}
			The following equations hold on $H_c(G,E)$:
			\begin{itemize}
				\item[(a)] $s_k zs_k = s_k(z) = \omega^k\oz$ and $s_z\oz s_k = s_k(\oz) = \omega^{-k}z$.
				\item[(b)] $\sigma(n)z = \oz\sigma(n+1)$ and $\sigma(n)\oz = z\sigma(n-1)$, for all $n \in \bZ$.
				\item[(c)] $[\zeta,z] = 1 - \sigma(0)$.
				\item[(d)] $[\oze,z] = \sigma(1)$.
				\item[(e)] $[H_z,z] = z - \tfrac{1}{2}(z\sigma(0) + \oz\sigma(1))$.
				\item[(f)] $[H_{\oz},z] = \tfrac{1}{2}(z\sigma(0) + \oz\sigma(1))$.
			\end{itemize}
		\end{Lemma}
		
		\begin{proof}
			For item (a), let $\theta_j = \pi j/m$. It is straightforward to compute
			\[
			2ia_j = -e^{-i\theta_j}(z - e^{2i\theta_j}\oz),
			\]
			from which
			\begin{align*}
				s_j(z) = z - 2\lpi\alpha_j,z\rpi a_j &= z + 2ia_j(\cos(\theta_j) + i\sin(\theta_j))\\
				&= e^{2i\theta_j}\oz.
			\end{align*}
			The computation for $s_j(\oz)$ is similar. Item (b) is immediate from item (a). For (c),  
			\[
			[\zeta,z] = \tfrac{1}{2}([\xi,x] + [\eta,y]) = 1 - \sum_{j=1}^{m}c(s_j)(\sin^2\theta_j + \cos^2\theta_j)s_j,
			\]
			while for (d) we have
			\begin{align*}
				[\oze,z] &= \tfrac{1}{2}([\xi,x] - [\eta,y]) + i[\xi,y]\\
				&= -\sum_{j=1}^{m}c(s_j)(\sin^2\theta_j - \cos^2\theta_j - 2i\sin\theta_j\cos\theta_j)s_j\\
				&= \sum_{j=1}^{m}c(s_j)e^{i2\theta_j}s_j.
			\end{align*}
			For (e), using (b) and (c) we compute
			\begin{align*}
				[H_z,z] = \tfrac{1}{2}[\{z,\zeta\},z]&= \tfrac{1}{2}\{z,[\zeta,z]\}\\
				&= \tfrac{1}{2}(z(1 - \sigma(0)) + (1 - \sigma(0))z)\\
				&= z - \tfrac{1}{2}(z\sigma(0) + \oz\sigma(1)).
			\end{align*}
			Item (f) is similar, using (b) and (d).
		\end{proof}
		
		\begin{Remar}
			The formulas for $[H_\oz,\oz]$ and $[H_{z},\oz]$ are obtained from items (e) and (f) of the previous proposition using the involution $\theta$ (complex conjugation) of $H_c(G,E)$.
		\end{Remar}
		
		\begin{Propo}\label{p:adZ0}
			For all $n\in \bZ$, we have the following commutation relations in $H_c(G,E)$:
			\begin{align*}
				[Z_0,z^n] &= z^n(-n + \sigma(0)) + \oz^n\sigma(n) + 2\sum_{j=1}^{n-1}\oz^jz^{n-j}\sigma(j)\\
				[Z_0,\oz^n] &= \oz^n(n - \sigma(0)) - z^n\sigma(-n) - 2\sum_{j=1}^{n-1}z^j\oz^{n-j}\sigma(-j).
			\end{align*}
		\end{Propo}
		\begin{proof}
			The result will be proven by induction. For $n=1$, we use Lemma \ref{l:dihprelim}(b), (e) and (f) to get $
			[Z_0,z] = z(-1 + \sigma(0)) + \oz\sigma(1).$ Assuming for $n$, we compute 
			\[
			[Z_0,z^n]z = -nz^{n+1} + z^n\oz\sigma(1) + 2\left(\sum_{j=1}^{n-1}\oz^{j+1}z^{n-j}\sigma(j+1)\right) + 
			\oz^{n+1}\sigma(n+1) .
			\]
			Thus, using $[Z_0,z^{n+1}] = [Z_0,z^n]z + z^{n}[Z_0,z]$, we get
			\[
			[Z_0,z^{n+1}] = z^{n+1}(-n-1 + \sigma(0)) +  2\left(\sum_{j=1}^{n}\oz^{j}z^{n-j+1}\sigma(j)\right)
			+ \oz^{n+1}\sigma(n+1).
			\]
			The other formula is obtained similarly  by using the complex conjugated versions of Lemma \ref{l:dihprelim}(e) and (f).
		\end{proof}
		
		\begin{Remar}
			Note that when $c=0$, we have $\sigma{(n)} = 0$, for all $n\in\bZ$ and we obtain the classical relations $[Z_0,z^n] = -nz^n$ and $[Z_0,\oz^n] = n\oz^n$.
		\end{Remar}
		
		\begin{Propo}\label{p:eigenvaluesdihedralchar}
			Let $\chi$ be a one-dimensional representation of $G$ and $c$ real-valued and close enough to $c=0$ such that $\pi:H_c(G,E)\to \End(M_c(\chi))$ is irreducible and unitarisable. For each $n\geq 0$, the eigenvalues of $Z_0$ on the space of Dunkl-harmonics of homogeneous degree $n$ are
			\[
			\lambda_n^{\pm} = \pm\sqrt{(n - N_c(\chi) - \chi(\sigma(n)))(n-N_c(\chi) + \chi(\sigma(n)))}.
			\]
		\end{Propo}
		
		\begin{proof}
			Let $\bC_\chi$ be the representation space of the character $\chi$. Since $M_c(\chi) = \bC[E]\otimes\bC_\chi \cong \bC[E]$, a general element in the degree $n$ part $M_c(\chi)_n$ is realized by a polynomial $\varphi$ in $\bC[E] = \bC[z,\oz]$ homogeneous of degree $n$ that can be written as
			\begin{equation}\label{e:nhompoly}
				\varphi = az^n + b\oz^n + z\oz\varphi_+
			\end{equation}
			with $a,b\in\bC$ and $\varphi_+$ of homogeneous degree $n-2$. We note that $z\oz\varphi_+$ is in the image of the operator $E_+$. Since there is a unitary decomposition of the space $\bC_n$, of homogeneous  polynomials of (total) degree $n$, $\bC_n = \mathcal{H}_c(\chi)_n \oplus E_+(\bC_{n-2})$, a polynomial $\varphi$ as in (\ref{e:nhompoly}) can only be harmonic if the coefficients $a,b$ are not simultaneously zero.
			From Proposition \ref{p:adZ0}, and using~\eqref{e:chisigma}, we have 
			\begin{align*}
				\pi(Z_0)(z^n) &= z^n(-n + N_c(\chi)) + \oz^n\chi(\sigma(n)) + \phi_+\\
				\pi(Z_0)(\oz^n) &= -z^n\chi(\sigma(n)) + \oz^n(n - N_c(\chi))  + \phi_+,
			\end{align*}
			where, in each formula, $\phi_+$ denotes an element in the image of $E_+$. Thus, the equation $\pi(Z_0)(\varphi) = \lambda\varphi$ implies that $\lambda$ must be an eigenvalue of the matrix
			\[
			\begin{pmatrix}
				-n+N_c(\chi) & -\chi(\sigma(n))\\
				\chi(\sigma(n)) & n-N_c(\chi)
			\end{pmatrix}.
			\]
			Since the spectrum of this diagonalisable matrix is $\lambda_n^\pm$, we are done.
		\end{proof}
		
		\begin{Remar}
			If, instead of $Z_0$ we consider an operator $\tilde Z_0 = Z_0 + \epsilon\sigma$, the eigenvalues of $\tilde Z_0$ on $\mathcal{H}_c(\chi)_n$ would be the spectrum of the matrix
			\[
			\begin{pmatrix}
				-n+ N_c(\chi) & - \chi(\sigma(n))(1-\epsilon)\\
				\chi(\sigma(n))(1+\epsilon) & n-N_c(\chi)
			\end{pmatrix},
			\]
			in which case we would have
			\[
			\tilde\lambda_n^{\pm} = \pm\sqrt{(n- N_c(\chi))^2 - \chi(\sigma(n))^2(1-\epsilon^2))}.
			\]
		\end{Remar}
		

		\begin{Propo}\label{p:eigenvaluesdihedralrefl}
			Let $\tau=\rho^{(u)}$ be a two-dimensional representation of $G$ and $c$ real-valued and close enough to $c=0$ such that $\pi:H_c(G,E)\to \End(M_c(\tau))$ is irreducible and unitarisable. For each $n\geq 0$, the eigenvalues of $Z_0$ on the space of Dunkl-harmonics of homogeneous degree $n$ are
			\begin{align*}
				\lambda_n^{\pm} &= \pm\sqrt{(n - \sigma_{n+u})(n + \sigma_{n+u}) }\\
				\mu_n^{\pm} &= \pm\sqrt{(n - \sigma_{n-u})(n + \sigma_{n-u}) }.
			\end{align*}
		\end{Propo}
		
		\begin{proof}
			We have $\tau = \rho^{(u)}$, with $1\leq u\leq q-1$. The method is the same as Proposition \ref{p:eigenvaluesdihedralchar}. We take $z,\oz$ also as basis for $V(\tau)$. Note that
			$\rho^{(u)}(\sigma(n))(z) = \sigma_{n+u}\oz$ and $\rho^{(u)}(\sigma(n))(\oz) = \sigma_{n-u}z$. Using Proposition \ref{p:adZ0}, we compute
			\begin{align*}
				\pi(Z_0)(z^n\otimes z) &= -n(z^n\otimes z) + 
				\sigma_{n+u}(\oz^n\otimes \oz) + \sigma_{u}(z^n\otimes \oz)  + \phi_+, \\
				\pi(Z_0)(\oz^n\otimes \oz) &= n(\oz^n\otimes \oz) - 
				\sigma_{-n-u}(z^n\otimes z) - \sigma_{-u}(\oz^n\otimes z)  + \phi_+,\\
				\pi(Z_0)(z^n\otimes \oz) &= -n(z^n\otimes \oz) + 
				\sigma_{n-u}(\oz^n\otimes z) + \sigma_{-u}(z^n\otimes z)  + \phi_+,\\
				\pi(Z_0)(\oz^n\otimes z) &= n(\oz^n\otimes z) - 
				\sigma_{-n+u}(z^n\otimes \oz) - \sigma_{u}(\oz^n\otimes \oz)  + \phi_+,
			\end{align*}
			where, in each formula, $\phi_+$ denotes an element in the image of $E_+$. It follows that the spectrum of $Z_0$ on $\mathcal{H}_c(\rho^{(u)})_n$ is the spectrum of the matrix
			\[
			\begin{pmatrix}
				-n &-\sigma_{n+u} & 0 & 0 \\
				\sigma_{n+u} & n & 0 & 0  \\
				0 & 0 & -n & -\sigma_{n-u} \\
				0 &0 & \sigma_{n-u} & n  
			\end{pmatrix},
			\]
			where we used that $\sigma_u=0$ for $1\leq u \leq q-1$, by Lemma~\ref{p:sigmas}. We conclude that the eigenvalues are the two pairs indicated in the statement.
		\end{proof}
		
		\begin{Remar}
			Again, if instead of $Z_0$ we consider the more general operator $\tilde{Z}_0 = Z_0+\epsilon\sigma(0)$, the eigenvalues of $\tilde Z_0$ on $\mathcal{H}_c(\rho^{(u)})_n$ would be the spectrum of the matrix
			\[
			\begin{pmatrix}
				-n &-(1-\epsilon)\sigma_{n+u} & 0 & 0 \\
				(1+\epsilon)\sigma_{n+u} & n & 0 & 0  \\
				0 & 0 & -n & -(1-\epsilon)\sigma_{n-u} \\
				0 & 0 & (1+\epsilon)\sigma_{n-u} & n  
			\end{pmatrix},
			\]
			in which case we would have the pairs of eigenvalues
			\begin{align*}
				\tilde\lambda_n^{\pm} &= \pm\sqrt{n^2 +(\epsilon^2-1)\sigma_{n+u}^2 }\\
				\tilde\mu_n^{\pm} &= \pm\sqrt{n^2 +(\epsilon^2-1)\sigma_{n-u}^2}.
			\end{align*}
		\end{Remar}

		
		The eigenvectors of the matrices appearing in the preceding proofs give the coefficients of the polynomials used inside the projection operator in~\eqref{e:eigenvectorsdihedral}. Using the 
		commutation relations from Proposition~\ref{u11}, it is straightforward to check that when~\eqref{e:Proj} acts on a polynomial homogeneous of degree $k$, the image is in $\ker E_-$. The final result now follows because $Z_0 + \epsilon \sigma$ commutes with $E_+,E_-$.

		\section{Spinor unitary duality}\label{Section:Clifford}
		
		\subsection{Clifford algebra}\label{s:clifford}
		
		We consider the complex Clifford algebra $(\mathcal{C},B_\bC)$ associated with the space $E_\bC$ and the symmetric bilinear form $B_\bC$. Denoting the natural embedding $\iota \colon E_\bC \to \mathcal{C}$, a set of generators  for $\mathcal{C}$,  corresponding to the chosen orthonormal basis of $E$, is given by
		$e_j = \iota (\xi_j)$ for $j\in\{1,\dotsc,2n\}$. These satisfy the anticommutation relations $\{e_{j},e_k\}
		= 2B_\bC(\xi_j,\xi_k)=2\delta_{jk}$. Note that in \cite{BDES} the opposite sign convention in the anticommutation relations is used, though this yields isomorphic complex Clifford algebras.

		The group actions on $E$ and  $E_\bC$ are naturally extended also to the Clifford algebra $\mathcal{C}$ through the correspondence between bases. By means of the action of $J$ on $\mathcal{C}$, we have $
		J \cdot e_j = -e_{n+j}$ and  $J \cdot e_{n+j} = e_j$ for $j=1,\dotsc,n$.
		One defines
		\begin{equation}\label{fj} 
			f_j= \pi^+(e_j) = \frac12(e_j - i \,e_{n+j}),\qquad f_j^{\dagger}= \pi^-(e_j) = 
			\frac12(e_j + i\, e_{n+j}),
		\end{equation} 
		for $j=1,\dotsc,n$,
		which correspond to $f_j = \iota(\zeta_j)$ and $f_j^{\dagger}=\iota(\bar \zeta_j)$. One has $ J_\bC \cdot f_j = -i\,f_j$ and $J_\bC \cdot f_j^{\dagger} = i\,f_j^{\dagger}$. 
		They satisfy the Grassmann relations
		\begin{equation}\label{ff} \{f_j^{\dagger},f_k\} = \delta_{jk}, \qquad \{f_j,f_k\} = 0 = \{f_j^{\dagger},f_k^{\dagger}\} .
		\end{equation}
		The inverse relations are given by 
		\begin{equation}\label{inverse} 
			e_j = \begin{cases}
				f_j + f_j^{\dagger} & \text{ if } 1\leq j \leq n,\\
				i(f_{j-n} - f_{j-n}^{\dagger}) & \text{ if } n+1\leq j \leq 2n.
			\end{cases}
		\end{equation}
		
		If $A \in \mathrm{U}(n)$ acts on $(z_1,\dotsc,z_{n})^T$ through matrix multiplication, its action on $(f_1,\dotsc,f_n)^T$ is given through multiplication by $A^*$ and on $(f_1^{\dagger},\dotsc,f_n^{\dagger})^T$ through multiplication by $A$.
		
		An irreducible spin module or spinor space $\mathbb{S}$ is given by the Grassmann algebra $\bigwedge W^-$ generated by $\{f_j^{\dagger}\}_{j=1}^n$, as $W^-$ is a $B_{\bC}$-isotropic subspace of $E_{\bC}$. The alternative choice $\bigwedge W^+$ is obtained via complex conjugation. 
		On $\mathbb{S}$, the skew-symmetric Grassmann variables $\{f_j^{\dagger}\}_{j=1}^n$ act by exterior multiplication, while the Grassmann derivatives  $\{f_j\}_{j=1}^n$ act by interior multiplication.
		The spinor space can be realised explicitly  inside the Clifford algebra as the product 
		\begin{equation}\label{e:spinor}
			\left(\bigwedge f^{\dagger}\right)I \quad \text{with }I = (f_1f_1^{\dagger})(f_2f_2^{\dagger})\dotsm (f_nf_n^{\dagger}).
		\end{equation}
		The action then follows by Clifford algebra multiplication, using the anticommutation relations~\eqref{ff} and $f_j I= 0$ for $j=1\dotsc,n$. 
		
		We will denote $\mathbb{S}^l$ for the ``degree'' $l$ subspace of $\mathbb{S}$, that is, the $l$-th exterior power of $\{f_j^{\dagger}\}_{j=1}^n$. The element $\sum_j f_j^{\dagger}f_j $ acts as multiplication by $l$ on $\mathbb{S}^l$ and  it induces a $\mathbb{Z}$-grading structure on the Clifford algebra because of the following simple result. 
		\begin{Propo}
			Denote $\beta = \sum_j f_j^{\dagger}f_j$. For all $1\leq k \leq n$, we have
			\[
			[\beta , f^\dagger_k] = f^\dagger_k, \qquad\qquad [\beta , f_k] = -  f_k.
			\]
			In particular, the Clifford algebra $\mathcal{C}$ is a $\mathbb{Z}$-graded algebra.
		\end{Propo}
		
		\begin{proof}
			The equations in the statement follow from a straightforward computation using the Grassmann relations (\ref{ff}). From this, we get $\mathcal{C} = \oplus_{k\in\mathbb{Z}} \mathcal{C}_k$, in which $\mathcal{C}_k$ is the $k$-eigenspace of $\textup{ad}(\beta)$ acting on $\mathcal{C}$. More specifically, $\mathcal{C}_k$ consists of all linear combination of monomials $f^\dagger_{j_1}\cdots f^\dagger_{j_p}f_{j_1}\cdots f_{j_q}$ in $\mathcal{C}$ with $p-q = k$. Note that $\mathcal{C}_k=0$ if $k<-n$ or $k>n$.
		\end{proof}
		
		Note that taking the residue mod 2 of the $\mathbb{Z}$-grading above 
		yields the usual $\mathbb{Z}_2$-grading of the Clifford algebra.

		The restriction of the Hermitian structure $(\cdot,\cdot)$ on $E_{\mathbb{C}}$ to $W^-$ naturally yields a Hermitian structure $(\cdot,\cdot)_\mathbb{S}$ on the spinors which on decomposable vectors of degree $p$ satisfies 
		\[
		(u_1^\dagger\wedge\cdots\wedge u_p^\dagger I,w_1^\dagger\wedge\cdots\wedge w_p^\dagger I)_\mathbb{S} = \det((w_k^\dagger,u_j^\dagger))_{j,k=1}^p,
		\]
		where $u_1^\dagger,\ldots , w_p^\dagger$ are elements in $W^-$. It is straightforward to check that, for all $1\leq j \leq n$ and $u^\dagger,w^\dagger\in \mathbb{S}$, the Clifford action satisfies the property
		\[
		(f_j^\dagger\cdot u^\dagger,w^\dagger)_\mathbb{S} =  (u^\dagger, f_j\cdot w^\dagger)_\mathbb{S}.
		\]

		\subsection{Tensor product and reflection group}\label{s:reflgroupclifford}
		
		We consider the tensor product $ H_c(G,E) \otimes \mathcal C$. To lighten the notation, we  shall mostly omit the tensor product symbols for elements of $ H_c(G,E) \otimes \mathcal C$ when it is clear in which part each factor resides. We naturally extend the involution $\ast$ on $E_\bC\oplus E_\bC^*$ to an anti-involution on $H_c(G,E)\otimes \mathcal{C}$. 
		
		Though the group actions on $E$ are naturally extended to  $\mathcal{C}$, the elements of $\bC[G]\subset H_c(G,E)$ do not interact with $\mathcal{C}$ in the tensor product $ H_c(G,E) \otimes \mathcal C$.  However, in $\mathcal{C}$ there is a realisation of a Pin group, a double cover of the orthogonal group associated with the used bilinear form, that yields a twisted action on $E_{\bC}\subset\mathcal{C}$. Denoting the covering map by $p\colon \mathrm{Pin} \to \mathrm{O}$, by viewing $G$ as a subgroup of the orthogonal group $\mathrm{O}$, one defines the preimage $\tilde G = p^{-1}(G)$ and similarly
		$(G^J)^{\widetilde{}}=p^{-1}(G^J)$ for the subgroup $G^J\subset G$, see also~\cite{Morris} and its sequels.

		\begin{Remar} 
			In a real Clifford algebra, depending on the sign convention in the defining relations, and the bilinear form being positive-definite or negative-definite, there is a realisation of either Pin$_+$ or Pin$_-$, where the preimage of a reflection has order 2 or order 4, and which are thus non-isomorphic groups in general.  
			
			In the complex Clifford algebra, both groups can be realised.   Below, the preimage of a reflection will have order 2. For them to have order 4, one can multiply by the imaginary unit $i$.
		\end{Remar}
		
		The group $\tilde G$ is realised in $\mathcal{C}$ as follows. 
		For a reflection $s\in\mathcal{S}$, the associated root $\alpha_s\in E$ is embedded in $\mathcal C$ as 
		\[
		\iota(\alpha_s  )
		= \sum_{j=1}^n ( \langle x_j , \alpha_s  \rangle e_j
		+ \langle  y_j , \alpha_s \rangle e_{n+j})
		= \sum_{j=1}^n ( \langle  z_j , \alpha_s \rangle f_j
		+\langle \bar z_j , \alpha_s \rangle f_j^{\dagger} 
		)
		\rlap{\,.}
		\]
		The elements $\pm \iota(\alpha_s  )/ \|\alpha_s\|$, with $\|\alpha_s\| =\sqrt{B(\alpha_s,\alpha_s)}$, then correspond to the preimages $p^{-1}(s)$ in $\tilde G$ and generate the realisation of $\tilde G$ in $\mathcal{C}$.  
		
		Now, we are interested in another realisation of $\tilde G$ in $H_c(G,E) \otimes \mathcal C$, which interacts with both $H_c(G,E) $ and $ \mathcal C$. Denote $\rho \colon \tilde G \to  H_c(G,E) \otimes \mathcal C \colon 
		\tilde g \mapsto
		\rho(\tilde g) =   p(\tilde g) \otimes \tilde g
		$. 
		This realisation of $\tilde G$ is 
		generated by  the elements $\pm s \otimes \iota(\alpha_s)/ \|\alpha_s\| \in H_c(G,E) \otimes \mathcal C$ for $s\in \mathcal{S}$. 
		For $\tilde g\in \tilde G$ and $X \in H_c(G,E) \otimes \mathcal C$, homogeneous elements for the $\mathbb{Z}_2$-grading inherited from the Clifford algebra $\mathcal C$, we have that 
		\begin{equation}\label{e:rhotg}
			\rho(\tilde g) X = (-1)^{|\rho(\tilde g)| |X|} (p(\tilde g)\cdot X )\rho(\tilde g),
		\end{equation}
		where $|\cdot|$ denotes the $\mathbb{Z}_2$-grading and $p(\tilde g)\cdot X$ is the action of $p(\tilde g) \in G$ as induced from the action of $G$ on $E$ to $H_c(G,E)$ as well as $\mathcal C$.  

		Next, we will introduce certain elements depending on $s\in\mathcal{S}$ that will be relevant for the Lie superalgebra realization in the next section. For a reflection $s\in\mathcal{S}$, the root vector can be written as $\alpha_s = \pi^+(\alpha_s) + \pi^-(\alpha_s)$, by means of the projection operators $\pi^{\pm} = (1\pm i J_{\mathbb{C}} )/2$ onto $E_\bC = W^+\oplus W^-$. 
		%
		Since $\iota(W^-) = \iota(\pi^-(E_\bC))$ in $\mathcal C$ is spanned by $\{f_1^\dagger,\ldots, f_n^\dagger\}$, we denote the projections of the roots onto $W^-$ in $\mathcal C$ by
		\begin{equation}\label{e:alphadagger}
			\alpha_s^\dagger = \iota(\pi^-(\alpha_s)) = \pi^-(\iota(\alpha_s)) = \sum_j\lpi \oz_j,\alpha_s \rpi f^\dagger_j.
		\end{equation}
		For $s\in\mathcal{S}$, define  $(\alpha_s^\dagger)^\vee\in \iota(W^+)$ as
		\begin{equation}\label{e:asd}
			(\alpha_s^\dagger)^\vee = 2 \iota(\pi^+(\alpha_s))/\|\alpha_s\|^2  =  \frac{2}{\|\alpha_s\|^2} \sum_{k=1}^n \lpi z_k, \alpha_s \rpi  f_k =   \sum_{k=1}^n \lpi \alpha_s^\vee,2\bar\zeta_k \rpi  f_k\rlap{\,,}
		\end{equation}
		and then define
		\begin{equation}\label{e:taus}
			\tau_s 
			= 1 -  \alpha_s^\dagger  (\alpha_s^\dagger)^\vee\in \mathcal C.
		\end{equation}
		
		\begin{Propo}\label{p:reflection}
			For $s\in\mathcal{S}$,  the element $\tau_s $ acts by the reflection with respect to the element $\alpha_s^\dagger\in W^-$ on $\mathbb{S}\cong \bigwedge W^-$. 
		\end{Propo}
		
		\begin{proof}
			From the realisation in (\ref{e:spinor}), one computes the action on a generator in $\mathbb{S}^1$
			\[
			\tau_s(f_j^\dagger I) 
			= f_j^\dagger I - \frac{2 (f_j^\dagger,\alpha_s^\dagger)}{(\alpha_s^\dagger,\alpha_s^\dagger)}\;\alpha_s^\dagger I,
			\]
			where, abusing the notation, $(f_j^\dagger,\alpha_s^\dagger) = (\pi^-(\xi_j),\pi^-(\alpha_s)) = B_\mathbb{C}(\pi^-(\xi_j),\pi^+(\alpha_s))$
			and $(\alpha_s^\dagger,\alpha_s^\dagger)  = (\pi^-(\alpha_s),\pi^-(\alpha_s)) = B_\mathbb{C}(\pi^-(\alpha_s),\pi^+(\alpha_s) )$.  It follows that $\tau_s$ acts as the desired
			reflection on $\mathbb{S}^1$. For $p>1$, choose a basis $\{w_1^\dagger,\ldots,w_n^\dagger\}$ of $W^-$ such that $w_1^\dagger = \alpha_s^\dagger$ and the rest is a basis for the kernel of the functional $(\alpha_s^\dagger)^*$ on $W^-$. Then, it is straight-forward to check for any basis element of $\mathbb{S}^p$, we have
			\[
			\tau_s(w_{i_1}^\dagger\wedge\cdots\wedge w_{i_p}^\dagger I) = w_{i_1}^\dagger\wedge\cdots\wedge w_{i_p}^\dagger I
			\]
			if $i_j > 1$ for all $j = 1, \ldots, p$ and
			\[
			\tau_s(w_1^\dagger\wedge\cdots\wedge w_{i_p}^\dagger I) = w_{i_1}^\dagger\wedge\cdots\wedge w_{i_p}^\dagger = -w_1^\dagger\wedge\cdots\wedge w_{i_p}^\dagger I,
			\]
			if $i_1 = 1.$ This finishes the proof.
		\end{proof}
		
		By means of the elements~\eqref{e:taus}, we define 
		\begin{equation}\label{e:Omega}
			\Omega_c = \sum_sc(s)s\otimes \tau_s \in H_c(G,E) \otimes \mathcal C\rlap{\,.}
		\end{equation}
		Extending the complex-conjugation automorphism $\theta$ of $E_\bC$ to a
		conjugate-linear involution on $H_c(G,E)\otimes \mathcal{C}$, we note that $\theta(\Omega_c) = -\Omega_c$. On the other hand, with respect to the anti-involution $\ast$ on $H_c(G,E)\otimes \mathcal{C}$, we have $\Omega_c^\ast = \Omega_c$.

		

		\subsection{Lie superalgebra}
		
		In the classical case, there are nine generators of the $\mathrm{U}(n)$-invariants in the tensor product of the Weyl and Clifford algebra, and they yield a realisation of the real form $\mathfrak{u}(2\vert1)$ of the Lie superalgebra $\mathfrak{gl}(2\vert1)$. We show that this realisation is preserved in the deformation $H_c(G,E) \otimes \mathcal{C}$. The classical realisation is obtained by setting $c=0$. Note that in~\cite{BDES} the central element $\beta - \mathbb{E}_{z^c} + \mathbb{E}_{z}$ (using their notation) can be added to the realization of $\mathfrak{sl}(2\vert1)$ to get $\mathfrak{gl}(2\vert1)$.

		\begin{Thm}\label{u21}
			We have a realisation of the real form $\mathfrak{u}(2\vert1)$ of the Lie superalgebra $\mathfrak{gl}(2\vert1)$,
			as a $\ast$-subalgebra of $H_c(G,E) \otimes \mathcal{C}$.
			It is generated by the odd elements
			\begin{align*}
				F_+ & = \sum_{j=1}^n \bar z_j  \otimes f_j^{\dagger}\rlap{\,,} & 
				\bar{F}_+ & = \sum_{j=1}^n z_j \otimes f_j\rlap{\,,} &
				F_- &  = 2\sum_{j=1}^n \zeta_j\otimes f_j^{\dagger} \rlap{\,,}  &
				\bar{F}_- &  = 2\sum_{j=1}^n   \bar \zeta_j \otimes f_j\rlap{\,.}
			\end{align*}
			The even subalgebra consists of the elements $E_{\pm} \otimes 1$, $H \otimes 1$, and 
			\begin{equation}\label{Z}
				\begin{aligned}
					Z_1	 & =   Z_0 - \Omega_c,\\
					Z_2 & =  \frac12 \sum_{j=1}^n [f_j^{\dagger},f_j] = \sum_{j=1}^n  f_j^{\dagger}f_j - \frac{n}{2}\rlap{\,,}
				\end{aligned}    
			\end{equation}
			using the definitions of Proposition~\ref{u11}, and~(\ref{e:Omega}) for $\Omega_c$. 
			
			The combination $Z_1-Z_2$ is central and denoting $ Z  := 2\,Z_2 -  Z_1 $,
			we have the following relations for the $\mathfrak{sl}(2\vert1)$-subalgebra:
			\begin{align}\label{evenrel}
				[H, E_{\pm}] &= \pm 2E_{\pm}	\rlap{\,,}
				& [E_+ ,E_- ] &= H\rlap{\,,}
				& [Z,H] &= 0 \rlap{\,,}
				&	[Z,E_{\pm}  ] & = 0 \rlap{\,.} 
			\end{align}
			%
			\begin{align}\label{oddrel1}
				\{F_{\pm}  ,F_{\pm}  \} &= 0\rlap{\,,} 
				& \{\bar{F}_{\pm}  ,\bar{F}_{\pm}  \}&= 0\rlap{\,,} 
				& \{F_{\pm}  ,\bar{F}_{\pm}  \} & = \pm 2E^{\pm}\rlap{\,,} \\
				\{F_{+}  ,F_{-}  \} & = 0\rlap{\,,} \label{oddrel2}
				& \{\bar{F}_{+}  ,\bar{F}_{-}  \} & = 0\rlap{\,,} 
				& \{F_{\pm}  ,\bar{F}_{\mp}  \} & = H\mp Z  \rlap{\,.} 		
			\end{align}
			For the interaction between the even and odd elements, we have
			\begin{align}\label{relevenodd2}
				[E_{\pm}  ,F_{\mp}  ] &= -F_{\pm}\rlap{\,,} 
				& [E_{\pm}  , \bar{F}_\mp  ] &= -\bar{F}_{\pm}\rlap{\,,}
				&[E_{\pm}  ,F_{\pm}  ] &= 0\rlap{\,,}
				& [E_{\pm}  , \bar{F}_{\pm}  ]	&= 0\rlap{\,,}
				\\  \label{relevenodd}
				[H,F_{\pm}  ] &= \pm   F_{\pm}\rlap{\,,} & 
				[H, \bar{F}_{\pm}  ]
				& = \pm   \bar{F}_{\pm}\rlap{\,,} & 
				[Z,F_{\pm}  ] &= F_{\pm}\rlap{\,,}  & 
				[Z, \bar{F}_{\pm}  ]& =  -\bar{F}_{\pm}\rlap{\,.} 
			\end{align}
			We have the $\ast$-relations $F_\pm^* = \bar{F}_{\mp}$, $E_{\pm}^* = -E_{\mp}$, $H^*=H$, $Z_1^*=Z_1$, $Z_2^*=Z_2$.
		\end{Thm}
		
		In terms of orthogonal vectors
		$\epsilon_1,\epsilon_2,\delta_1$, dual
		to the Cartan algebra elements $(Z_1+H)/2$, $(Z_1-H)/2$, and $Z_2$, so $\epsilon_1+\epsilon_2-\delta_1$ is dual to the central element $Z_1-Z_2$, 
		the associated root vectors are
		\begin{align*}
			\epsilon_1 - \epsilon_2 &\leftrightarrow E_{+}\,,&
			-\epsilon_2 +\delta_1  &\leftrightarrow F_{+}\,,&
			\epsilon_1 - \delta_1 &\leftrightarrow \bar F_{+}\,,
			\\
			-\epsilon_1 + \epsilon_2 &\leftrightarrow E_{-}\,,&
			\epsilon_2- \delta_1  &\leftrightarrow \bar F_{-}\,,&
			-\epsilon_1+\delta_1   &\leftrightarrow F_{-}\,.
		\end{align*}
		\begin{proof} The relations~\eqref{evenrel} follow from Proposition~\ref{u11} and the fact that $Z_2$ commutes with $H$ and $E_{\pm}$ as their Clifford algebra part is 1, together with $Z_1-Z_2$ being central in $\mathfrak{u}(2\vert1)$, which we show at the end of the proof.
			
			The first two (pairs of) relations of \eqref{oddrel1} follow immediately using~\eqref{comm} and~\eqref{ff}. 
			For the last pair we compute
			\begin{align*}
				\{F_{+}  ,\bar{F}_{+}  \} & =  \sum_{j=1}^n\sum_{k=1}^n z_j \bar z_k\{ f_j   ,      f_k^{\dagger}\}
				=  \sum_{j=1}^n z_j \bar z_j
				= 2E_+ \rlap{\,,}
				\\
				\{F_{-}  ,\bar{F}_{-}  \} & = 4 \sum_{j=1}^n\sum_{k=1}^n \bar \zeta_j \zeta_k\{ f_j   ,      f_k^{\dagger}\} 
				= 4 \sum_{j=1}^n \bar \zeta_j \zeta_j
				= -2 E_- \rlap{\,.}
			\end{align*}

			For the first two relations of~\eqref{oddrel2}, we have
			\[
			\{F_{+}  , F_{-}  \}  = 2 \sum_{j=1}^n\sum_{k=1}^n \{\bar z_jf_j^{\dagger} , \zeta_k f_k^{\dagger}\} 
			=   2\sum_{j=1}^n\sum_{k=1}^n [\bar z_j,\zeta_k] f_j^{\dagger}f_k^{\dagger} = 0	\rlap{\,,}
			\]
			which vanishes because by Lemma~\ref{Lemma1.2}, $[\bar z_j,\zeta_k]$ is symmetric in $j$ and $k$, while $f_j^{\dagger}f_k^{\dagger}=-f_k^{\dagger}f_j^{\dagger}$ is antisymmetric. In the same manner, 
			\[
			\{\bar{F}_{+}  ,\bar  F_{-}  \}  =  2\sum_{j=1}^n\sum_{k=1}^n \{z_jf_j ,\bar \zeta_k f_k\} 
			= 2 \sum_{j=1}^n\sum_{k=1}^n [z_j,\bar \zeta_k] f_jf_k	 \rlap{\,.}
			\]
			
			The relations~\eqref{relevenodd2} follow directly from~\eqref{comm} and~\eqref{xp}.
			
			The first two relations of~\eqref{relevenodd} follow from~\eqref{Euler}. The interaction with $Z$ follows by means of the other relations and the super Jacobi identity,
			\[
			[Z,F_{\pm}  ] = [\pm H \mp \{F_{\pm}  , \bar{F}_{\mp}\},  F_{\pm} ] 
			= \pm ( \pm   F_{\pm}) 
			\mp\frac12 [   \bar{F}_{\mp} ,  \{F_{\pm},F_{\pm}\}] 
			=   F_{\pm}\rlap{\,,}
			\]
			and similarly $[Z,\bar F_{\pm}  ] = [\mp H \pm \{F_{\mp}  , \bar{F}_{\pm}\}, \bar F_{\pm} ] =\mp(\pm\bar{F}_{\pm})= -\bar F_{\pm}  $. 
			
			This leaves only the last set of relations of~\eqref{relevenodd} to be proved. 
			Using $f_k  f_j^{\dagger} =  \delta_{jk} - f_j^{\dagger}f_k$, we have
			\begin{equation}\label{e:anticom1}
				\{F_{+}  ,\bar{F}_{-}  \}  = 2\sum_{j=1}^n\sum_{k=1}^n( \bar z_j \bar \zeta_k    f_j^{\dagger}f_k 
				+    \bar \zeta_k \bar z_j f_k        f_j^{\dagger}    ) 
				=   2\sum_{j=1}^n \bar \zeta_j\bar z_j      
				-  2\sum_{j=1}^n\sum_{k=1}^n [\bar \zeta_k,\bar z_j] f_j^{\dagger}f_k  \rlap{\,,}
			\end{equation}
			and similarly
			\begin{equation}\label{e:anticom2}
				\{F_{-}  ,\bar{F}_{+}  \} 
				= 2\sum_{j=1}^n\sum_{k=1}^n (\zeta_jz_k f_j^{\dagger}f_k
				+ z_k \zeta_jf_k  f_j^{\dagger})  
				= 2 \sum_{j=1}^n z_j \zeta_j   
				+ 2 \sum_{j=1}^n\sum_{k=1}^n [\zeta_j,z_k] f_j^{\dagger}f_k \rlap{\,.}
			\end{equation}
			Using the relation~$[\zeta_j,z_k] = [\bar \zeta_k,\bar z_j]$ from  Lemma~\ref{Lemma1.2}, and the computation~\eqref{H},  we arrive at $ \{F_{+}  ,\bar{F}_{-}  \}+\{F_{-}  ,\bar{F}_{+}  \} =   2H$, as defined in Proposition~\ref{u11}. 
			
			Next, we show $\{F_{-}  ,\bar{F}_{+}  \} - \{F_{+}  ,\bar{F}_{-}  \} =   2Z$. By means of relation~\eqref{RC}, we have
				\begin{align*}
					\sum_{j=1}^n\sum_{k=1}^n [\bar \zeta_k,\bar z_j] \otimes f_j^{\dagger}f_k 
					& =    \sum_{j=1}^n  f_j^{\dagger}f_j - \sum_{s\in\mathcal S} c(s)s 
					\bigg(\sum_{j=1}^n\lpi \bar z_j,\alpha_s  \rpi f_j^{\dagger} \bigg) 
					\bigg(\sum_{k=1}^n \lpi \alpha_s^\vee,\bar\zeta_k \rpi  f_k\bigg) \rlap{\,,}
					\\   & =    \sum_{j=1}^n  f_j^{\dagger}f_j - \frac12\sum_{s\in\mathcal S}c(s)s\,  \alpha_s^\dagger(\alpha_s^\dagger)^\vee  \rlap{\,.}
				\end{align*}
				where we used~\eqref{e:asd}. 
				Hence, using (\ref{e:anticom1}), (\ref{e:anticom2}) and 
				\[
				\sum_{j} \oze_j\oz_j = \sum_{j} \oz_j\oze_j + n - \sum_s c(s)s\rlap{\,,}
				\] 
				which follows by~\eqref{RC}, we get
				\[
				\frac12(\{F_{-}  ,\bar{F}_{+}  \} - \{F_{+}  ,\bar{F}_{-}  \})
				=2\sum_{j=1}^n  f_j^{\dagger}f_j -n - Z_0  + \Omega_c\rlap{\,.}
				\]
				
				Finally, we show that $Z_1 - Z_2$ is central. From the adjoint action of $ \sum_j f_j^{\dagger}f_j $ on the Clifford algebra $\mathcal C$ we obtain
				\begin{equation}\label{e:commZ2}
					[ Z_2 , F_{\pm} ] =  F_{\pm} , \qquad [ Z_2 , \bar F_{\pm} ] = - \bar F_{\pm}\rlap{\,.}
				\end{equation}
				Comparing with the last two relations of \eqref{relevenodd}, and since $Z =Z_2 -( Z_1 - Z_2 ) $, we conclude that $Z_1 - Z_2$
				commutes with $F_{\pm},\bar F_{\pm}$. Note that this also implies the same adjoint action of $Z_1$:
				\begin{equation}\label{e:commZ1}
					[ Z_1 , F_{\pm} ] =  F_{\pm} , \qquad [ Z_1 , \bar F_{\pm} ] = - \bar F_{\pm}\rlap{\,.}
				\end{equation}
				The $\ast$-relations follow by straight-forward application of the  anti-involution $\ast$ on $H_c(G,E)\otimes \mathcal{C}$.
			\end{proof}

			
			The $\mathfrak{u}(2\vert1)$ central element $Z_1 - Z_2$ can be written also in terms of elements in the (super)centraliser of the $\mathfrak{spo}(2\vert1)= \mathfrak{osp}^{sk}(1\vert2) \cong \mathfrak{osp}(1\vert2)$ subalgebra generated by the two odd elements $D_+ := F_+ + \bar F_+ $ and $D_- := F_- + \bar F_- $. The latter satisfy the following relations, with $H$ as defined in Proposition~\ref{u11}, 
			\begin{equation}\label{spo12}
				\{D_+,D_-\} = 2H, \qquad  [H, D_{\pm} ] = \pm D_{\pm}\rlap{\,.}
			\end{equation}
			In~\cite{DOV}, it was obtained that for distinct $j,k\in \{1,\dotsc,N\}$ 
			\begin{equation}\label{Ojk}
				O_{jk}   = x_j \xi_k - x_k \xi_j + \frac12 e_j e_k + \sum_{s\in\mathcal S}  c(s)\frac{s\,  \iota(\alpha_s)}{\|\alpha_s\|^2}  (  e_j \langle x_k , \alpha_s  \rangle   - e_k\langle x_j , \alpha_s  \rangle  )
			\end{equation}
			is in the $\mathfrak{spo}(2\vert1)$ (super)centraliser.
			
			For $j\in \{1,\dotsc,n\}$ and $k=n+j$, the elements of the form~\eqref{Ojk} can be written in terms of the complex variables as follows
			\[
			i \, O_{j,n+j} = (\bar z_j \bar \zeta_j - z_j \zeta_j ) + \frac12 \,-\,f_j^{\dagger}f_j   -\sum_{s\in\mathcal S}  c(s)\frac{s\,  \iota(\alpha_s)}{\|\alpha_s\|^2}  ( \langle\bar z_j, \alpha_s\rangle   f_j^{\dagger} 
			-\langle z_j , \alpha_s \rangle f_j 
			)\rlap{\,.} 
			\]    
			Summing over $j\in \{1,\dotsc,n\}$, and comparing with the definitions~\eqref{Z}, we obtain 
			\[
			Z_1 - Z_2 = i\sum_{j=1}^n O_{j,n+j}\rlap{\,.} 
			\]
			


			\subsection{Spinor decomposition} 
			
			
			Recall that the action of $G$ on $E$ induces an action on $H_c(G,E) \otimes \mathcal C$ by outer automorphisms, since the copy of $G$ inside $H_c(G,E)$ does not interact with $\mathcal C$. 
			However, the realisation of the double cover $\tilde G$ given by $\rho\colon \tilde G\to H_c(G,E) \otimes \mathcal C$ does provide an inner action, see the discussion in Section~\ref{s:reflgroupclifford}.
			
			\begin{Propo}\label{rhoGJ}
				The $\mathfrak{u}(2\vert1)$ Lie superalgebra realisation of Theorem~\ref{u21} is invariant under the action of the group $G^J= G \cap \mathrm{U}(n)$, and supercommutes with the realisation of 
				$(G^J)^{\widetilde{}}=p^{-1}(G^J)$ in $H_c(G,E) \otimes \mathcal{C}$. 
			\end{Propo}
			\begin{proof}
				Besides the previously introduced notations~\eqref{notat}, consider the $n\times 1$ column matrices
				\[
				f = \begin{pmatrix} f_1  \\ \vdots  \\  f_n\end{pmatrix} , 
				\quad f^{\dagger} = \begin{pmatrix} f^{\dagger}_1 \\ \vdots  \\  f^{\dagger}_n\end{pmatrix}.
				\]
				Using tensor products as multiplication, we can write 
				$
				F_+  =     \bar z^T f^{\dagger} = z^* f^{\dagger} 
				$, and the other odd elements have similar expressions. 
				
				The action of an element of $ g\in G^J$ on $z$, $\bar \zeta$ and $f^{\dagger}$ corresponds to multiplication by a unitary matrix $A$,  and in turn the action on $\bar z$, $\zeta$ and $f$ will be given by the conjugate transpose matrix $A^*$. In this way, we find for $F_+$: 
				\[
				g \cdot F_+ 
				=  \left( A  z\right)^* A f^{\dagger}
				=    z^* A^*  A f^{\dagger}
				= F_+\rlap{\,.}
				\]
				The invariance of the central element $Z_1 - Z_2$ follows from the invariance of on the one hand $Z_2 = (f^{\dagger})^T f - n/2$, using the same argument as above, and on the other hand of $Z = (\{F_{+}  ,\bar{F}_{-}  \}-\{F_{-}  ,\bar{F}_{+}  \} )/2$. 
				
				The supercommuting property then follows by~\eqref{e:rhotg}.
			\end{proof}

			


			We consider the space $K_c(\tau) = M_c(\tau) \otimes \mathbb{S}$. We endow it with the Hermitian structure $(\cdot|\cdot)$ given as the product of the Hermitian structures $\beta_{c,\tau}(\cdot,\cdot)$ and $(\cdot,\cdot)_\mathbb{S}$ on $M_c(\tau)$ and on $\mathbb{S}$, as described in Sections \ref{s:unitary} and \ref{s:clifford}, respectively.
			
			\begin{Propo}
				The pair $(K_c(\tau),(\cdot|\cdot))$ is a $\ast$-Hermitian module for the Lie superalgebra $\mathfrak{u}(2\vert1)$ inside $H_c(G,E)\otimes\mathcal{C}$.
			\end{Propo}
			
			\begin{proof}
				Using the contravariance properties of the forms $\beta_{c,\tau}(\cdot,\cdot),(\cdot,\cdot)_\mathbb{S}$ on $M_c(\tau)$ and on $\mathbb{S}$
				we get that $K_c(\tau)$ is a $\ast$-Hermitian $H_c(G,E)$ module. Since 
				$\mathfrak{u}(2\vert1)$ is a $\ast$-subalgebra, the claim follows.
			\end{proof}
			
			Furthermore, since $Z_1$ commutes with $E_-$ and with $Z_2$, it induces an operator on the spaces $\mathcal{H}_c(\tau)_k^l = \mathcal{H}_c(\tau)_k\otimes \mathbb{S}^l$,  
			where $\mathcal{H}_c(\tau)_{k}$ is the intersection of the degree-$k$ part of $M_c(\tau)$ with the kernel of $E_-$, and $\mathbb{S}^l$ is the space of spinors of degree $l$. 
			Let $\mathscr{S}_c(\tau)_k^l$ denote the set
			of eigenvalues of $Z_1$ acting on  $\mathcal{H}_c(\tau)_k^l$.
			By means of $Z_1^* = Z_1$ and $\theta(Z_1) = - Z_1$, similar to Proposition \ref{p:ZnaughtProp}, we have
			
			\begin{Propo}\label{e:ZProp}
				Assume $c$ is real-valued and sufficiently close to $c=0$ such that $K_c(\tau)$ is unitarisable. Then, the operator $Z_1$ restricted to $\mathcal{H}_c(\tau)_k^l$ is diagonalisable for all $k\geq 0$ and $0\leq l \leq n$. Moreover, $\mathscr{S}_c(\tau)_k^l\subseteq \bR$, and if $\nu\in\mathscr{S}_c(\tau)_k^l$, then $-\nu$ is also in $\mathscr{S}_c(\tau)_k^l$.
			\end{Propo}
			
			The weights for the $\mathfrak{u}(2\vert1)$-realisation can be written in terms of the total polynomial degree ($H$-eigenvalue), the spinor degree ($Z_2$-eigenvalue) and the $Z_1$-eigenvalue, which, when $c=0$, is the difference between anti-holomorphic and holomorphic degree. We denote by $\mathcal{H}_c(\tau)_{k,\nu}^l$ the $\nu$-eigenspace of the operator $Z_1$ inside $\mathcal{H}_c(\tau)_{k}^l$ and by $\mathcal{M}_c(\tau)_{k,\nu}^l$ the intersection of $\mathcal{H}_c(\tau)_{k,\nu}^l$ with the kernel of both $F_-$ and $\bar{F}_-$. To ease notation, we will occasionally omit the subscript $c$ and the reference to $\tau$ in what follows, i.e., $\mathcal{H}_{k,\nu}^l =\mathcal{H}_c(\tau)_{k,\nu}^l$ and $\mathcal{M}_{k,\nu}^l =\mathcal{M}_c(\tau)_{k,\nu}^l$.
			
			\begin{Propo}\label{p:MongDecomp}
				Assume the parameter map $c$ is chosen such that $K_c(\tau)$ is unitarisable. Then, 
				the space $\mathcal{H}_c(\tau)_{k,\nu}^l$ decomposes as an orthogonal sum
				\[
				\mathcal{H}_{k,\nu}^l = \mathcal{M}_{k,\nu}^l \oplus F_+ \mathcal{M}_{k-1,\nu-1}^{l-1} \oplus\bar{F}_+\mathcal{M}_{k-1,\nu+1}^{l+1} \oplus ( \lambda_1 F_+\bar{F}_+ - \lambda_2 \bar{F}_+F_+) \mathcal{M}_{k-2,\nu}^l ,
				\]
				with $\lambda_1= k-2 + 2n -N_c(\tau) + \nu -l $ the eigenvalue of $H-Z-2$ acting on $\mathcal{H}_c(\tau)_{k,\nu}^l$ and $\lambda_2= k-2 -N_c(\tau) - \nu +l$ the eigenvalue of $H+Z-2$ acting on $\mathcal{H}_c(\tau)_{k,\nu}^l$.
				
				Note that, when $k=0$, only the first summand of the right-hand side is non-zero, while, for $k=1$, the fourth summand is always zero, for $l=0$, the second summand is zero, and for $l=n$, the third summand is zero. 
			\end{Propo}
			\begin{proof}
				Parts of the proof are similar to the proofs of~\cite[Proposition~1]{BDES} and of \cite[Proposition~5.7]{CDM}.
				
				Using the star relations in Theorem~\ref{u21} and the Hermitian structure $(\cdot|\cdot)$, the summands in the right-hand side are readily shown to be orthogonal. 
				
				By means of the commutation relations~\eqref{oddrel2} and~\eqref{relevenodd2}, it follows that each summand is in $\ker E_-$, proving the inclusion $\supseteq$. For instance, for $M\in \mathcal{M}_{k-2,\nu}^l$,
				\begin{align*}
					E_- ( \lambda_1 F_+\bar{F}_+ - \lambda_2 \bar{F}_+F_+) M 
					& = 
					( \lambda_1 (F_+E_- -F_-)\bar{F}_+ - \lambda_2 (\bar{F}_+E_- - \bar F_- )F_+) M  . \\
					& = 
					( -\lambda_1 \{ F_-,\bar{F}_+\}
					+ \lambda_2  \{\bar F_- ,F_+\}) M ,  \\
					& = 
					( -\lambda_1 (H+Z)
					+ \lambda_2  (H-Z)) M   .
				\end{align*}
				This vanishes because the action of $H-Z$ (resp.~$H+Z$) on $ \mathcal{M}_{k-2,\nu}^l$ is given by the eigenvalue $\lambda_1$ (resp.~$\lambda_2$).

				To prove the other inclusion $\subseteq$, let $ H_{k,\nu}^l \in \mathcal{H}_{k,\nu}^l$. 
				We will define four elements, corresponding to the four summands in the right-hand side, that sum to $ H_{k,\nu}^l$. 
				For each of these elements that we will define, the appropriate eigenvalues for $H$, $Z_1$ and $Z_2$ will follow by means of the action on $ H_{k,\nu}^l $
				and the commutation relations~\eqref{relevenodd}, \eqref{e:commZ2}, and \eqref{e:commZ1}.
				
				First, we define 
				\begin{equation}\label{e:M1}
					{M}_{k-2,\nu}^l  =  
					\begin{cases}
						0 & \text{if } F_-\bar F_- H_{k,\nu}^l =0, \\
						(F_-\bar F_- H_{k,\nu}^l)/(\lambda_1\lambda_2(\lambda_1+\lambda_2+2)) &\text{if } F_-\bar F_- H_{k,\nu}^l \neq0   
						.
					\end{cases}
				\end{equation}
				If $F_-\bar F_- H_{k,\nu}^l \neq 0$, then, by means of the relations $\{F_-,F_-\} = 0 = \{\bar F_- , \bar F_-\}$ and $\{F_-,\bar F_-\}  = -2E_- $, we immediately find that 
				$
				F_-\bar F_- H_{k,\nu}^l = - \bar F_-F_- H_{k,\nu}^l
				$ is killed by both $F_-$ and $\bar F_-$.
				Hence, it follows that $
				F_-\bar F_- H_{k,\nu}^l \in \mathcal{M}_{k-2,\nu}^l
				$. 
				Note that the denominator in the definition~\eqref{e:M1} is non-zero whenever $F_-\bar F_- H_{k,\nu}^l \neq0   $, and $c$ is chosen such that $K_c(\tau)$ is unitarisable. Indeed, for $0 \neq M\in \mathcal{M}_{k-2,\nu}^l$ with $k\geq 2$ and $0\leq l<n$ (these restrictions are satisfied in the case at hand, as otherwise $F_-\bar F_- H_{k,\nu}^l = 0$), we have $F_+ M \neq 0$ and, using the unitarity of $(K_c(\tau),(\cdot|\cdot))$, together with $F_+^* = \bar F_-$ and $\bar F_- M = 0$, we obtain
				\[
				0 < ( F_+ M \vert F_+M ) = ( \bar F_-F_+ M \vert M ) = ( \{F_+, \bar F_-\} M \vert M )= \lambda_1  ( M \vert M )\rlap{\,.}
				\]
				In a similar manner, for $0 \neq M\in \mathcal{M}_{k-2,\nu}^l$ with $k\geq 2$ and $0 <l\leq n$, using now  $\bar F_+ M \neq 0$, one obtains $\lambda_2 >0$.
				
				Next, if $\lambda_1(\lambda_1+\lambda_2+2) \bar{F}_+{M}_{k-2,\nu}^l=\bar F_- {H}_{k,\nu}^l$, then define ${M}_{k-1,\nu-1}^{l-1} =  0$, while otherwise define
				\begin{equation}\label{e:M2}
					{M}_{k-1,\nu-1}^{l-1} =   
					(\bar F_- {H}_{k,\nu}^l - \lambda_1(\lambda_1+\lambda_2+2) \bar{F}_+{M}_{k-2,\nu}^l)/(\lambda_1 +2) \rlap{\,.}
				\end{equation}
				Using a similar argument as above, the denominator is shown to be non-zero whenever $\lambda_1(\lambda_1+\lambda_2+2) \bar{F}_+{M}_{k-2,\nu}^l\neq\bar F_- {H}_{k,\nu}^l$,
				and $c$ is chosen such that $K_c(\tau)$ is unitarisable. We now show that $ {M}_{k-1,\nu-1}^{l-1} \in \mathcal{M}_{k-1,\nu-1}^{l-1}  $.
				From $\{\bar F_+ , \bar F_-\}= 0$ follows $\bar F_-{M}_{k-1,\nu-1}^{l-1} =0 $.
				Using $\{F_-,\bar F_+\} = H+Z$ and $F_- {M}_{k-2,\nu}^l = 0$, we have 
				\begin{align*}
					(\lambda_1 +2) F_-{M}_{k-1,\nu-1}^{l-1} 
					& =  F_-\bar F_-  {H}_{k,\nu}^l - \lambda_1(\lambda_1+\lambda_2+2)(H+Z) {M}_{k-2,\nu}^l 
					\\ & =  F_-\bar F_-  {H}_{k,\nu}^l - \lambda_1(\lambda_1+\lambda_2+2)\lambda_2 {M}_{k-2,\nu}^l \rlap{\,,}
				\end{align*}
				which vanishes by definition of ${M}_{k-2,\nu}^l$. 
				
				Next, if  $\lambda_2(\lambda_1+\lambda_2+2) F_+{M}_{k-2,\nu}^l=- F_- {H}_{k,\nu}^l$, then define ${M}_{k-1,\nu+1}^{l+1} =  0$, while otherwise define
				\begin{equation}\label{e:M3}
					{M}_{k-1,\nu+1}^{l+1} = (F_- {H}_{k,\nu}^l+\lambda_2 (\lambda_1+\lambda_2+2) F_+ {M}_{k-2,\nu}^l)/(\lambda_2 +2)  \rlap{\,.}    
				\end{equation}
				Using a similar argument as above, the denominator is shown to be non-zero whenever $\lambda_2(\lambda_1+\lambda_2+2) F_+{M}_{k-2,\nu}^l\neq- F_- {H}_{k,\nu}^l$,
				and $c$ is chosen such that $K_c(\tau)$ is unitarisable. We now show that $ {M}_{k-1,\nu+1}^{l+1} \in \mathcal{M}_{k-1,\nu+1}^{l+1}  $. 
				From $\{F_+,F_-\} = 0 $ follows $ F_-{M}_{k-1,\nu+1}^{l+1} = 0  $.
				Using $\{\bar F_-, F_+\} = H-Z$,
				\begin{align*}
					(\lambda_2 +2) \bar F_-{M}_{k-1,\nu+1}^{l+1}
					& = \bar F_-F_- {H}_{k,\nu}^l + \lambda_2 (\lambda_1+\lambda_2+2)(H-Z){M}_{k-2,\nu}^l
					\\ &= -F_-\bar F_- {H}_{k,\nu}^l + \lambda_2(\lambda_1+\lambda_2+2)\lambda_1{M}_{k-2,\nu}^l \rlap{\,,}
				\end{align*}
				which vanishes by definition of ${M}_{k-2,\nu}^l$. 
				
				Finally, defining 
				\begin{equation}\label{e:M4}
					{M}_{k,\nu}^l = {H}_{k,\nu}^l  - ( \lambda_1 F_+\bar{F}_+ - \lambda_2 \bar{F}_+F_+) {M}_{k-2,\nu}^l - F_+ {M}_{k-1,\nu-1}^{l-1} - \bar F_+{M}_{k-1,\nu+1}^{l+1} \rlap{\,,}    
				\end{equation}
				it is clear that $H_{k,\nu}^l$ can be written in terms of the summands~\eqref{e:M1}, \eqref{e:M2}, \eqref{e:M3}, \eqref{e:M4} as
				\[
				{H}_{k,\nu}^l = {M}_{k,\nu}^l + F_+ {M}_{k-1,\nu-1}^{l-1} +\bar{F}_+{M}_{k-1,\nu+1}^{l+1} + ( \lambda_1 F_+\bar{F}_+ - \lambda_2 \bar{F}_+F_+) {M}_{k-2,\nu}^l \rlap{\,.}
				\]
				We conclude the proof by showing that ${M}_{k,\nu}^l$ is in the kernel of both $F_-$ and $\bar{F}_-$:
				\begin{align*}
					F_- {M}_{k,\nu}^l 
					&= F_- {H}_{k,\nu}^l  + ( \lambda_1 F_+(H+Z) + \lambda_2 (H+Z)F_+) {M}_{k-2,\nu}^l  - (H+Z) {M}_{k-1,\nu+1}^{l+1} \rlap{\,,}
					\\ & = F_- {H}_{k,\nu}^l  + ( \lambda_1\lambda_2  + \lambda_2 (\lambda_2+2))F_+ {M}_{k-2,\nu}^l  - (\lambda_2 +1+2-1) {M}_{k-1,\nu+1}^{l+1} \rlap{\,,}
					\\ & = F_- {H}_{k,\nu}^l  + \lambda_2( \lambda_1  +\lambda_2+2)F_+ {M}_{k-2,\nu}^l  - (\lambda_2 +2) {M}_{k-1,\nu+1}^{l+1} \rlap{\,,}
				\end{align*}
				vanishes by definition of ${M}_{k-1,\nu+1}^{l+1}$, and
				\begin{align*}
					\bar F_- {M}_{k,\nu}^l 
					& =\bar F_- {H}_{k,\nu}^l  - ( \lambda_1 (H-Z)\bar F_+ + \lambda_2  \bar F_+(H-Z)) {M}_{k-2,\nu}^l  - (H-Z) {M}_{k-1,\nu-1}^{l-1} \rlap{\,,}
					\\ & = F_- {H}_{k,\nu}^l  - ( \lambda_1(\lambda_1+2)  + \lambda_2\lambda_1 ) \bar F_+ {M}_{k-2,\nu}^l  - (\lambda_1 +1+2-1) {M}_{k-1,\nu-1}^{l-1} \rlap{\,,}
					\\ & = F_- {H}_{k,\nu}^l  - \lambda_1( \lambda_1  +\lambda_2+2)\bar F_+ {M}_{k-2,\nu}^l  - (\lambda_1 +2) {M}_{k-1,\nu-1}^{l-1} \rlap{\,,}
				\end{align*}
				vanishes by definition of $ {M}_{k-1,\nu-1}^{l-1}$.
			\end{proof}

			Denoting by $\mathcal{R}$ the centralizer algebra of $\mathfrak{u}(2|1)$ inside $H_c(G,E)\otimes \mathcal{C}$, we arrive at the following:
			
			\begin{Cor}\label{cor:genSpinorDecomp}
				For real-valued parameters $c$ such that $K_c(\tau)$ is unitarisable, we have
				the following joint $(\mathcal{R},\mathfrak{u}(2|1))$-decomposition 
				\[
				K_c(\tau) = \bigoplus_{l=0}^{n}\bigoplus_{k\in \mathbb{Z}_{\geq 0}}\bigoplus_{\nu\in\mathscr{S}_c(\tau)_k^l} 
				\mathcal{M}_c(\tau)_{k,\nu}^l\otimes L_{\mathfrak{u}(2|1)}(k + n - N_c(\tau),\nu,l-\tfrac{n}{2}),
				\]
				where $L_{\mathfrak{u}(2|1)}(\lambda,\nu,\gamma)$ denotes the irreducible lowest weight representation  of $\mathfrak{u}(2|1)$ of $(H,Z_1,Z_2)$ lowest weight $(\lambda,\nu,\gamma)$.
			\end{Cor}
			
			\begin{proof}
				By means of the $\mathfrak{su}(1,1)$-decomposition~\eqref{e:sl2decomp}, we have 
				\[
				K_c(\tau) = \bigoplus_{l=0}^{n} M_c(\tau) \otimes  \mathbb{S}^l = \bigoplus_{l=0}^{n} \bigoplus_{k\in\mathbb{Z}_{\geq 0}}\mathcal{H}_c(\tau)_{k}^l\otimes L_{\mathfrak{su}(1,1)}(k + n - N_c(\tau)).
				\]
				Denoting by $\mathcal{H}_c(\tau)_{k,\nu}^l$ the $\nu$-eigenspace of the operator $Z_1$ inside $\mathcal{H}_c(\tau)_{k}^l$, we refine the decomposition to 
				\[
				K_c(\tau) = \bigoplus_{l=0}^{n} \bigoplus_{k\in\mathbb{Z}_{\geq 0}} \bigoplus_{\nu \in \mathscr{S}_c(\tau)_k^l}\mathcal{H}_c(\tau)_{k,\nu}^l\otimes L_{\mathfrak{u}(1,1)}(k + n - N_c(\tau),\nu),
				\]
				where $L_{\mathfrak{u}(1,1)}(\mu,\nu)$ denotes the irreducible lowest weight module for the $\mathfrak{u}(1,1)$-algebra realised by $E_+,E_-,H,Z_1$, with 
				$(H,Z_1)$-lowest weight $(\mu,\nu)$. The claim now follows from Proposition \ref{p:MongDecomp}.
			\end{proof}
			
			\section{Example: dihedral groups in the spinor case}\label{Section:CliffordDihedral}
			

			We again consider $E=\bR^2$ with $G=I_2(m)$ a dihedral group, and use the same notation as in Section~\ref{Section:dihedral}. The Clifford algebra has two generators $e_1,e_2$ corresponding to the real basis $\xi,\eta$ of $E$. 
			As $n=1$, we omit the subscripts from the complex linear combinations~\eqref{fj}. Hence, $f$ and $f^{\dagger}$ satisfy the anticommutation relations $\{f,f\} = 0 = \{f^{\dagger},f^{\dagger}\}$ and $\{f,f^{\dagger}\}=1$. The spinor space $\mathbb{S} \cong \bigwedge f^{\dagger}$ is two-dimensional with basis $1,f^{\dagger} $, or $ff^{\dagger},f^{\dagger} $ when realised inside $\mathcal{C}$. This provides a matrix representation for $f$ and $f^{\dagger}$, in which $e_1$ and $e_2$ correspond to the first two Pauli matrices.
			
			Let $\mathcal{M}_c(\tau)_{k}^l$ denote the subspace of $K_c(\tau) = M_c(\tau) \otimes \mathbb{S}$ with elements of polynomial degree $k$ (eigenvalue for $H$), spinor degree $l$ (eigenvalue for $Z_2$), and that are in the kernel of both $F_-$ and $\bar F_-$ (and hence also of $E_-$). In the low-dimensional case $n=1$, we show that the third $\mathfrak{u}(2\vert 1)$-weight, being the eigenvalue for $Z_1$ or $Z$, is determined by the previous two weights. The $\mathfrak{u}(2\vert 1)$-representations occurring in the decomposition of $K_c(\tau)$ are all atypical. 
			To find the weights, computations similar to the ones in Section~\ref{Section:dihedral} can be used. Specifically in the case $n=1$, another approach is also possible. 
			
			\begin{Propo}\label{Prop41}
				For $M_{k}^l\in \mathcal{M}_c(\tau)_{k}^l$, we have the following actions
				\begin{align*}
					H \, M_{k}^l & =  \left(k + 1 - N_c(\tau) \right)M_{k}^l\rlap{\,,}\\
					Z_2\, M_{k}^l & =   (-1)^{l+1} \frac12 \,M_{k}^l\rlap{\,,}\\
					Z_1\, M_{k}^l & =   (-1)^{l}\left(k  - N_c(\tau)\right) M_{k}^l\rlap{\,,}\\
					Z\, M_{k}^l & =   (-1)^{l+1}\left(k +1 - N_c(\tau)\right) M_{k}^l\rlap{\,.}
				\end{align*}
			\end{Propo}
			\begin{proof}
				The first two actions follow immediately from Lemma~\ref{Lemma1.4} and the Clifford algebra anticommutation relations~\eqref{ff} as $Z_2 = f^{\dagger} f - 1/2$. 
				
				For the final action, note that the $\mathfrak{u}(2\vert1)$ central element can be written as $Z_1-Z_2 = i O_{1,2}$, as defined in~\eqref{Ojk}. For $E=\bR^2$, this is related~\cite[p.13]{DOV} to the Scasimir element of the $\mathfrak{spo}(2\vert1)$ subalgebra generated by $D_+ = F_+ + \bar F_+ $ and $ D_- = F_- + \bar F_- $:
				\[
				i O_{1,2} =  i \frac12( D_-D_+ - D_+D_- - 1) \, e_1 e_2 = -( D_-D_+ - D_+D_- - 1) Z_2\rlap{\,,}
				\]
				where we used $i\,e_1e_2=(1-2f^{\dagger}f) $. Both the Scasimir and the pseudoscalar $e_1e_2$ are in the ghost centre of $\mathfrak{spo}(2\vert1)$, their product is in the centre and squares, up to a constant, to the $\mathfrak{spo}(2\vert1)$ Casimir. 
				In a manner similar to, for instance, Section 5.2 of \cite{DOV2}, we obtain
				\begin{align*}
					Z_1\, M_{k}^l 
					& =   -(D_-D_+ - D_+D_- -2) Z_2 M_{k}^l\rlap{\,,}\\
					& =   -(2H-2) (-1)^{l+1} \frac12 \, M_{k}^l\rlap{\,,}\\
					& =   (-1)^{l} \left(k  - N_c(\tau) \right)  M_{k}^l\rlap{\,.}
				\end{align*}
				The action of $Z= 2\,Z_2 -Z_1$ now follows from the computed actions. 
			\end{proof}
			
			We will briefly discuss how this result relates to the ones of Section~\ref{Section:dihedral}, namely Theorem~\ref{t:eigenvaluesdihedral} and the subsequent Corollaries, and use these to give a basis for $ \mathcal{M}_c(\tau)_{k}^l$. 
			
			As $\mathbb{S} = \mathbb{S}^0 \oplus \mathbb{S}^1$, we have  
			\begin{equation*}
				K_c(\tau) =  M_c(\tau) \otimes  \mathbb{S}^0 \oplus  M_c(\tau) \otimes \mathbb{S}^1.
			\end{equation*}
			When $n=1$, using the explicit expressions for the root vectors~\eqref{rootsdihedral}, we can write 
			\[
			Z_1 =  \bar  z\bar \zeta  -z \zeta +    2 \sigma Z_2   \rlap{\,.}
			\]
			As $2Z_2 \mathbb{S}^l = (-1)^{l+1} \mathbb{S}^l$, there is a  $\mathfrak{u}(1,1)$-subalgebra of $\mathfrak{u}(2\vert1)$ generated by $H,E_{\pm}$
			and $Z_1$ which will act on the space
			$ M_c(\tau) \otimes  \mathbb{S}^l $
			as a copy of the algebra $\mathfrak{u}(1,1)$ considered in Part I with central element 
			\[
			Z_1 |_{M_c(\tau) \otimes \mathbb{S}^l} = Z_0 + (-1)^{l+1} \sigma\rlap{\,.}
			\] 
			
			From the expressions~\eqref{e:eigenvaluesdihedral} in Theorem~\ref{t:eigenvaluesdihedral} for the eigenvalues of $Z_0 + \epsilon \sigma$, we see that they simplify to $\pm(k- N_c(\tau))$ for $\epsilon^2 = 1$. 
			The $\mathfrak{u}(1,1)$-decompositions in the corollaries following  Theorem~\ref{t:eigenvaluesdihedral} can be used to write
			\begin{equation*}
				K_c(\tau) =  \bigoplus_{l= 0}^1\bigoplus_{k\in \mathbb{Z}} \mathcal{H}_c(\tau)_{|k|,\sign(k)}\otimes L_{\mathfrak{u}(1,1)}(|k| + 1 - N_c(\tau),k- \sign(k) N_c(\tau))\otimes \mathbb{S}^l\rlap{\,,}
			\end{equation*}
			where $\mathcal{H}_c(\tau)_{|k|,\pm}$ denotes the $\pm(|k|- N_c(\tau))$-eigenspace of $Z_0+  (-1)^{l+1}\sigma$ in $\ker E_-$. 
			This can now be refined to a $\mathfrak{u}(2|1)$-decomposition by combining pairs of $\mathfrak{u}(1,1)$-modules into a single (atypical) representation of $\mathfrak{u}(2|1)$.
			
			For $l=0$, we have $(-1)^{l+1} = -1$ and hence for $k\in \mathbb{Z}_{>0}$, the eigenvectors~\eqref{e:eigenvectorsdihedral} simplify to
			\begin{align*}
				h(-,\sigma_{\tau,k})^{+} & = \Proj_k\left[ 2( k - N_c(\tau)  \bar z^k \otimes \bar  z(\sigma_{\tau,k}) - 2\sigma_{\tau,k}  z^k \otimes z(\sigma_{\tau,k})   \right] \rlap{\,,}\\
				h(-,\sigma_{\tau,k})^{-} & =  2( k - N_c(\tau) ) \Proj_k\left[ z^k \otimes  z(\sigma_{\tau,k})\right]\rlap{\,,}
			\end{align*}
			where $\sigma_{\tau,k} \in \sigma(\tau)_k$. 
			From the eigenvalues in Proposition~\ref{Prop41}, we observe that the element(s) $h(-,\sigma_{\tau,k})^{+} \otimes 1 \in \mathcal{H}_c(\tau)_{k,+}\otimes \mathbb{S}^0$ for $\sigma_{\tau,k} \in \sigma(\tau)_k$ form(s)
			a basis for $ \mathcal{M}_c(\tau)_{k}^0$. This basis is killed by both $\bar F_-$ (this is immediate from the action of $f$ on $\mathbb{S}^0$) and $F_-$, which is non-trivial and implies $h(-,\sigma_{\tau,k})^{+}\in\ker\zeta$ as was mentioned in Remark~\ref{r:harmDunkl}. The kernel of $\bar\zeta$ is treated in the next paragraph. 
			
			Similarly, for $l=1$, we have $(-1)^{l+1} = 1$ and hence for $k\in \mathbb{Z}_{>0}$, the eigenvectors~\eqref{e:eigenvectorsdihedral} simplify to
			\begin{align*}
				h(+,\sigma_{\tau,k})^{+} & = 2( k - N_c(\tau) ) \Proj_k\left[\bar  z^k \otimes  \bar z(\sigma_{\tau,k})  \right]\rlap{\,,}\\
				h(+,\sigma_{\tau,k})^{-} & = \Proj_k\left[ 2( k - N_c(\tau)) z^k \otimes  z(\sigma_{\tau,k}) - 2\sigma_{\tau,k} \bar z^k \otimes \bar  z(\sigma_{\tau,k}) \right] \rlap{\,,}
			\end{align*}
			where $\sigma_{\tau,k} \in \sigma(\tau)_k$.
			Again by Proposition~\ref{Prop41}, a basis for $ \mathcal{M}_c(\tau)_{k}^1$ is given by $h(+,\sigma_{\tau,k})^{-} \otimes f^{\dagger}\in \mathcal{H}_c(\tau)_{k,-}\otimes \mathbb{S}^1 $ for $\sigma_{\tau,k} \in \sigma(\tau)_k$. This basis is killed by both $F_-$ (this is immediate from the action of $f^{\dagger}$ on $\mathbb{S}^1$) and $\bar F_-$, which is non-trivial and implies $h(+,\sigma_{\tau,k})^{-}\in\ker\bar\zeta$.
			
			Next, from the eigenvalues in Proposition~\ref{Prop41}, we also have $F_- h(-,\sigma_{\tau,k})^{-} \otimes 1  \neq 0$ for $k\in \mathbb{Z}_{>0}$ and $\sigma_{\tau,k} \in \sigma(\tau)_k$. Since
			$F_- h(-,\sigma_{\tau,k})^{-} \otimes 1  $ is killed by both $F_-$ and $\bar F_-$,
			and moreover, using the commutation relations~\eqref{e:commZ1}, we have 
			\[
			Z_1 F_- h(-,\sigma_{\tau,k})^{-}  \otimes 1 
			=  -(k-1- N_c(\tau))F_- h(-,\sigma_{\tau,k})^{-} \otimes 1 \rlap{\,.}
			\]
			Thus, $F_- h(-,\sigma_{\tau,k})^{-} \otimes 1  $ is an element of $\mathcal{M}_c(\tau)_{k-1}^1$, and hence a linear combination of $h(+,\sigma_{\tau,k})^{-} \otimes f^{\dagger} $ for $\sigma_{\tau,k} \in \sigma(\tau)_k$ when $k>1$, or an element of the 
			space $1\otimes V(\tau) \otimes f^{\dagger} $ when $k=1$.
			In a similar manner, $\bar F_- h(+,\sigma_{\tau,k})^{+} \otimes f^{\dagger} $ is seen to be in $\mathcal{M}_c(\tau)_{k-1}^0$.
			
			In this way, the decomposition is refined to 
			\[
			K_c(\tau) = \bigoplus_{l=0}^1\bigoplus_{k\in \mathbb{Z}_{\geq0}} \mathcal{M}_c(\tau)_{k}^l \otimes L_{\mathfrak{u}(2\vert1)}(k + 1 - N_c(\tau),(-1)^{l}\left(k  - N_c(\tau)\right),l-\tfrac{1}{2}).
			\]
			Here, $\mathcal{M}_c(\tau)_{k}^l$ denotes the subspace with elements of polynomial degree $k$, spinor degree $l$, and that are in the kernel of both $F_-$ and $\bar F_-$, with a basis given explicitly in the preceding discussion. It is an irreducible module for the centraliser algebra $\mathrm{Cent}_{H_c(G,E) \otimes \mathcal{C}}(\mathfrak{u}(2\vert1))$ only when $\tau$ is one-dimensional. Note that in the present dihedral case, the eigenvalue of $Z_1$ is completely determined by the eigenvalues of $H$ and $Z_2$, and thus we have omitted the $Z_1$-eigenvalue $\nu$ from the notation in $\mathcal{M}_c(\tau)_{k,\nu}^l$, when compared with the general decomposition in Corollary \ref{cor:genSpinorDecomp}.

			\begin{Propo}
				The centraliser $(\mathrm{Cent}_{H_c(G,E) \otimes \mathcal{C}}(\mathfrak{u}(2\vert1))$ is generated by the central element $Z_1 - Z_2$ and $\rho((G^J)^{\widetilde{}}\,)$, where $G^J$ is abelian for $G$ dihedral.
			\end{Propo}
			\begin{proof}
				Similar to Proposition~\ref{p:cent}, using now that the (super)centraliser of the $\mathfrak{spo}(2\vert1)$-subalgebra is generated by $O_{12} = -i(Z_1-Z_2)$ and the double cover of $G$.
			\end{proof}


			
			
			\section*{Acknowledgements}
			
			DC and MDM were supported in part by the Engineering and Physical Sciences Research Council grant [EP/N033922/1] (2016). HDB and RO were supported in part by EOS Research Project number 30889451.
			RO was supported by a postdoctoral fellowship, fundamental research, of the Research Foundation -- Flanders (FWO), number 12Z9920N

			

		\end{document}